\documentclass[11pt]{amsart}
\usepackage[hmargin=2.5cm,top=2cm,bottom=2.5cm,includehead]{geometry}
\usepackage[colorlinks,citecolor=blue]{hyperref}
\usepackage{amsmath,amsthm,amsfonts,amssymb,amsmath,oldgerm}
\usepackage{mathrsfs}
\numberwithin{equation}{section}
\usepackage{mdframed}
\usepackage{fancyhdr}  
\usepackage{graphicx}
\usepackage{enumerate,color,xcolor}
\usepackage{stmaryrd}
\usepackage[normalem]{ulem}
\usepackage[inline]{enumitem}
\usepackage{algorithm}
\usepackage{algorithmic}
\usepackage{hyperref,bookmark}
\hypersetup{
    colorlinks,            
    citecolor=blue,     
    filecolor=blue,      
    urlcolor=blue,       
    linkcolor=blue,
}
\newtheorem{thm}{Theorem}[section]

\newtheorem{hyp}{Assumption}[section]
\newtheorem{lem}{Lemma}[section]
\newtheorem{rmk}{Remark}[section]
\newtheorem{cor}{Corollary}[section]
\newtheorem{prop}{Proposition}[section]
\renewcommand{\theequation}{\thesection.\arabic{equation}}

\makeatletter
\newcommand{\opnorm}{\@ifstar\@opnorms\@opnorm}
\newcommand{\@opnorms}[1]{%
  \left|\mkern-1.5mu\left|\mkern-1.5mu\left|
   #1
  \right|\mkern-1.5mu\right|\mkern-1.5mu\right|
}
\newcommand{\@opnorm}[2][]{%
  \mathopen{#1|\mkern-1.5mu#1|\mkern-1.5mu#1|}
  #2
  \mathclose{#1|\mkern-1.5mu#1|\mkern-1.5mu#1|}
}
\makeatother

\numberwithin{equation}{section}

\newcommand{\beq}{\begin{equation*}}
\newcommand{\eeq}{\end{equation*}}
\newcommand{\ben}{\begin{eqnarray}}
\newcommand{\een}{\end{eqnarray}}
\newcommand{\beno}{\begin{eqnarray*}}
\newcommand{\eeno}{\end{eqnarray*}}
\let\f=\frac

\renewcommand{\theequation}{\arabic{section}.\arabic{equation}}

\newcommand{\bit}{\begin{itemize}}
\newcommand{\eit}{\end{itemize}}

\newcommand{\N}{{\mathbb N}}
\newcommand{\Z}{{\mathbb Z}}
\newcommand{\R}{{\mathbb R}}

\newcommand{\pa}{\partial}

\newcommand{\ba}{\begin{aligned}}
\newcommand{\ea}{\end{aligned}}
 \def\na{\nabla}

\let\pa=\partial

\newcommand{\mP}{{\mathscr P}}

\def\cA{{\mathcal A}}
\def\cB{{\mathcal B}}
\def\cC{{\mathcal C}}
\def\cD{{\mathcal D}}
\def\cE{{\mathcal E}}

\def\cH{{\mathcal H}}
\def\cI{{\mathcal I}}
\def\cJ{{\mathcal J}}

\def\cL{{\mathcal L}}

\def\cP{{\mathcal P}}
\def\cQ{{\mathcal Q}}
\def\cR{{\mathcal R}}

\def\CC{\mathbf{C}}
\def\RR{\mathbf{R}}
\def\NN{\mathbf{N}}
\def\H1A{{\mathsf{H}^1_{\mathsf{A}}}}
\def\virgp{\raise 2pt\hbox{,}}
\def\cdotpv{\raise 2pt\hbox{;}}

\def\C{\mathop{\mathbb C\kern 0pt}\nolimits}
\def\DD{\mathop{\mathbb D\kern 0pt}\nolimits}
\def\EE{\mathop{{\mathbb E \kern 0pt}}\nolimits}
\def\K{\mathop{\mathbb K\kern 0pt}\nolimits}
\def\N{\mathop{\mathbb N\kern 0pt}\nolimits}
\def\Q{\mathop{\mathbb Q\kern 0pt}\nolimits}
\def\R{\mathop{\mathbb R\kern 0pt}\nolimits}
\def\SS{\mathop{\mathbb S\kern 0pt}\nolimits}
\def\<{\langle}
\def\>{\rangle}

\def\ls{\lesssim}
\def\S{\mathbb{S}}
\def\sss{\mathsf s}
\newcommand{\br}[1]{\left\langle {#1} \right\rangle}

\newcommand{\rr}[1]{\left( {#1} \right)}
\newcommand{\nr}[1]{\left| {#1} \right|}
\newcommand{\LLO}[1]{\| #1 \|_{L^1}}

\newcommand{\Lab}[3]{\left\| {#1} \right\|_{L^{#2}_{#3}}}
\newcommand{\Hab}[3]{\left\| {#1} \right\|_{H^{#2}_{#3}}}

\def\nn{{\mathrm n}}

\newcommand{\wei}{{\langle v \rangle}}
\newcommand{\weis}{{\langle v_* \rangle}}

\newcommand\ds{\displaystyle}
\newcommand\dD{\mathrm{d}}

\title[Numerical analysis of Landau equation]{Numerical analysis of the homogeneous Landau equation: approximation, error estimate and  simulation}

\keywords{Kinetic theory; Landau equation;  Fourier/Galerkin spectral method}

\subjclass[2010]{
  Primary: 65M15, 82C40, 35B65 
  Secondary: 35R09, 65M70, 65N35, 
}

\begin{document}

 \author{Francis Filbet}
 \address[Francis Filbet]{Institut de Math\'ematiques de Toulouse,
   Universit\'e de Toulouse, F-31062, France}
 \email{francis.filbet@math.univ-toulouse.fr}

 \author{Yanzhi Gui}
 \address[Yanzhi Gui]{Department of Mathematical Sciences, Tsinghua University, Beijing, 100084, P.R. China.}
 \email{guiyz22@mails.tsinghua.edu.cn}

 \author{Ling-Bing He}
 \address[Ling-Bing He]{Department of Mathematical Sciences, Tsinghua University, Beijing, 100084, P.R. China.}
 \email{hlb@tsinghua.edu.cn}

 %
 %
 %

\begin{abstract} We construct a numerical solution to the spatially homogeneous Landau equation with Coulomb potential on a domain $\cD_L$ with $N$ retained Fourier modes. By deriving an explicit error estimate in terms of $L$ and $N$, we demonstrate that for any prescribed error tolerance and fixed time interval $[0,T]$, there exist choices of $\cD_L$  and $N$ satisfying explicit conditions such that the error between the numerical and exact solutions is below the tolerance. Specifically, the estimate shows that sufficiently large $L$ and $N$ (depending on initial data parameters and $T$) can reduce the error to any desired level. Numerical simulations based on this construction are also presented. The results in particular demonstrate the mathematical validity of the spectral method proposed in the referenced literature.
\end{abstract}

\maketitle

\tableofcontents

\section{Introduction and main results}
\label{sec1}
\setcounter{equation}{0}
\setcounter{figure}{0}
\setcounter{table}{0}

The Landau equation is a fundamental kinetic equation that describes
the evolution of the distribution of charged particles in a
collisional plasma \cite{ref41}, particularly in regimes where grazing
collisions dominate \cite{ref20,Vi98}. Alongside the Boltzmann
equation, it is considered one of the most important equations in
kinetic theory. Recently, it has garnered renewed interest in
computational plasma physics due to its significant applications in
fusion reactors. The Landau equation governs the evolution of the
charged particles' mass distribution function, $f(t,x,v)$, in phase
space $(x,v)\in\Omega\times \R^3$, and is expressed as follows:
$$
\frac{\partial f}{\partial t} + v\cdot \nabla_x f \,=\, \cQ(f,f),
$$
while the bilinear collision operator $\cQ$ is defined as
\beq
\cQ(g,h)\,:=\,\nabla \cdot\left([a * g] \,\nabla h\,-\,[a * \nabla g] \,h\right),
\eeq
where the positive semidefinite matrix $a=(a_{ij})_{3\times3}$ takes the form
\ben
\label{eq:aij}
a_{i j}(z):=|z|^{\gamma+2} \,\Pi_{i j}(z), \qquad \Pi_{i j}(z)\,:=\,\delta_{i j}\,-\,\frac{z_i z_j}{|z|^2},
\een
with $-4\leq \gamma\leq 1$.  The most interesting case corresponds to
$ \gamma=-3$ (Coulomb case) associated with the physical interaction in plasmas,
which will be considered here. The case $\gamma=0$ is usually referred to as the Maxwellian case since the equation is reduced to a sort of degenerate linear Fokker-Planck equation preserving the same moments as the Landau equation \cite{ref55}.

This equation is derived from the Boltzmann equation by incorporating a cutoff Rutherford cross section within the framework of the weak-coupling regime. This groundbreaking work has laid the foundation for understanding the kinetic behavior of particle systems in plasmas and other related fields, emphasizing the importance of collective interactions among charged particles.

The primary objective of this work is to conduct a comprehensive numerical investigation of the spatially homogeneous Landau equation with Coulomb potential. This includes the development of a novel numerical approximation, derivation of corresponding error estimates, and detailed simulations. This study is significant because it enhances our understanding of the Landau equation, a fundamental model in kinetic theory that describes the dynamics of particle systems. Accurate numerical solutions are crucial for analyzing complex phenomena such as plasma behavior and statistical mechanics.

\subsection{Brief review on Landau equation}
\label{sec11}

Here, we consider the spatially homogeneous  Landau equation:
\begin{equation}\label{1}
\begin{cases}
\ds\partial_t f \,=\, \cQ(f,f)\,, \\[0.9em]
\ds f|_{t=0} = f_0.
\end{cases}
\end{equation}
The Landau equation provides crucial insights into the
dynamics of such systems, making it an essential tool for researchers
in statistical mechanics and plasma physics.  Solutions $f$ to \eqref{1} formally conserve mass, momentum, and energy,
that is,
$$
\int_{\R^3} f(t,v) \, \varphi(v) \, \dD v = \int_{\R^3} f_0(v) \,
\varphi(v) \, \dD v\,, \,\text{for}\quad \varphi(v) = 1,\, v_i,\,
\frac{|v|^2}2,\quad i\,=\,1,\,2,\,3.
$$
This follows from the weak formulation:
\begin{eqnarray*}
  &&\ds\int_{\R^3} \cQ(f,f) (v)\, \varphi(v) \, \dD v\,=\,
  \\
&& \ds - \frac12 \, \sum_{i=1}^3\sum_{j=1}^3 \iint_{\R^3 \times \R^3} a_{ij}(v-v_{*})\bigg\{ \frac{\partial_j f}{f}(v) - \frac{\partial_j f}{f}(v_{*})  \bigg\} \bigg\{ \partial_i \varphi(v) - \partial_i \varphi(v_{*})  \bigg\}   
                                             f(v) f(v_{*}) \, \dD v \,
       \dD v_{*}.
\end{eqnarray*}
Defining the macroscopic quantities: 
\begin{equation*} 
\rho_f:= \int_{\R^3} f(v) \, \dD v, \quad 
u_f := \frac{1}{\rho_f} \int_{\R^3} v \,f(v) \, \dD v, \quad
T_f := \frac{1}{3 \rho_f} \int_{\R^3} |v-u|^2\, f(v) \, \dD v,
\end{equation*}  
we introduce the relative entropy:   
$$
\cH(f\;|\;\mu_{f})(t) \,:=\, \int_{\R^3} \bigg(f(t,v) \log \bigg(
\f{f(t,v)}{\mu_{f}(v)} \bigg) -f(t,v)
+\mu_{f}(v)\bigg) \, \dD v,
$$
 where $\mu_{f}$ is the associated Maxwellian:
 $$
\mu_{f}(v) \,:=\, \frac{\rho_f}{(2\,\pi\, T_f)^{3/2}} \, \exp\left(-\frac{|v-u_f|^2}{2\,T_f}\right).
$$ 
  
  The formal $\cH$-Theorem states: 
$$
\frac{\dD}{\dD t}\cH(f\,|\;\mu_{f})(t) = -\cD(f(t,\cdot)) \le 0,
$$
with entropy dissipation:
\begin{eqnarray*}
\cD(f)  &:=& - \ds\int_{\R^3} \cQ(f,f)(v)\, \log f(v) \, \dD v
\\
&=&  \ds\frac12 \sum_{i,j=1}^3\iint_{\R^3 \times \R^3}
    a_{ij}(v-v_{*}) \left\{ \frac{\partial_i f}{f}(v) -
    \frac{\partial_i f}{f}(v_{*})  \right\} \left\{ \frac{\partial_j
    f}{f}(v) - \frac{\partial_j f}{f}(v_{*})  \right\}   \, f(v)\,
    f(v_{*}) \, \dD v \, \dD v_{*}.
    \end{eqnarray*}

In the following, we provide a brief review of previous work from both analytical and numerical perspectives. We focus on results that are closely related to our problem. On one hand, we refer to \cite{Guo} for the global result within the symmetric perturbation framework. This result was later extended to a more general setting in \cite{CM}. Furthermore, recent work \cite{HST} discusses the short-term existence of solutions with rough initial data. On the other hand, we cite \cite{DV1,DV2} for discussions on the existence, regularity, and global dynamics associated with hard potentials. In \cite{Vi98}, Villani established the global existence of $\mathcal{H}$-solutions for soft potentials. Desvillettes \cite{D15} demonstrated that $\mathcal{H}$-solutions are sufficiently smooth, confirming that they qualify as weak solutions in the conventional sense. Subsequently, in \cite{CDH,DHJ}, the authors explored the long-term behavior of $\mathcal{H}$-solutions and the associated entropy-energy functional. For the uniqueness results, \cite{FG} presents specific criteria, while \cite{HJL} provides a direct proof within the negative weighted Sobolev spaces. More recently, \cite{GS} established a monotonicity property for Fisher information, which leads to the existence of a global classical solution (see also \cite{GGL} for further extensions). Lastly, we highlight \cite{HJL} for insights into sharp regularity estimates when the initial data exhibit polynomial or exponential tails.

Regarding numerical approximation of the Landau equation, we first refer to finite difference methods initially developed in \cite{DD,BCDL,L98} and studied more recently in \cite{W17,W24}. Hermite spectral methods have also been proposed in \cite{LWW,LRW}. Nowadays, particle methods are increasingly popular, as discussed in \cite{F09} and more recent works such as \cite{CDW,CHWW,CJT,HW,IHW,DLXY}. Direct Simulation Monte Carlo (DSMC) methods are also applied \cite{Bi,BP}, and recently, neural network-based numerical approaches have gained attention \cite{LJH,MC,NLY}. A widely used approach is the Fourier-Galerkin spectral method \cite{PRT}, which leverages the convolutional structure of the collision integral for efficient implementation via Fast Fourier Transform (FFT). Specifically, the complexity of evaluating the collision operator at each time step is $O(N \log N)$, where $N$ represents the total number of velocity modes. Additional properties and applications of spectral methods, especially for inhomogeneous problems via time-splitting techniques, are discussed in \cite{BGZ,FP,PG,ZG}.

This approach introduces the concept of the {\it periodic Landau collision operator} to facilitate the application of the Fourier spectral method. Numerical solutions are obtained by retaining a finite number of Fourier modes, thereby reducing the problem to a system of ordinary differential equations (ODEs) for the Fourier coefficients.

It is important to note that, despite the significance of
deterministic methods for simulating Landau-type equations, there has
been limited work dedicated to the convergence analysis of these
schemes toward the exact solutions. Error estimates have, however,
been developed in \cite{FM, AGT, HQY} for the spectral method used to solve the Boltzmann equation. More recently, Pennie and Gamba \cite{PG} obtained an error estimate for the conservative spectral method approximation of the space-homogeneous Landau equation with hard potentials, where $0 \leq \gamma \leq 1$.

This work addresses a central challenge in the numerical analysis of
the spatially homogeneous Landau equation for Coulomb potentials,
which are of particular physical relevance. We propose a rigorous
analysis of high-order accurate numerical schemes, accompanied by
provable error estimates. Specifically, our goal is to develop
high-accuracy numerical methods that preserve the key structural
properties of the Landau equation ({\it e.g.}, non-negativity) and to derive
error bounds that quantify their convergence rates in suitable norms.

In the following we first introduce the  appropriate functionnal framework and defined the truncated problem, where
the Landau equation is restricted to a bounded velocity domain. Then
we provide an approximation based on a periodized procedure and the
use of Fourier spectral method for the Landau equation. Finally, we
state our main results on error estimates between the approximated solution and the exact solution to the Landau equation \eqref{1}. 

\subsection{Notations and function spaces}
\label{sec12}

We begin by introducing the key notations and function spaces used
throughout this work. We define the inner product $\langle ., . \rangle$ as
$$
\langle f, g \rangle := \int_{\mathbb{R}^3} f(v)g(v) \dD v
$$
and the  Japanese bracket is $\langle v \rangle :=
\sqrt{1+|v|^2}$. Then we define  the weighted Lebesgue  space
$L_l^p$, for $p\geq 1$ and $l\in\R$ as
$$
L_l^p \,:=\, \left\{ f: \R^3 \mapsto \R,  \,\quad \| f \|_{L_l^p} = \left( \int_{\mathbb{R}^3} |f(v)|^p \langle v \rangle^{lp} \, \dD v \right)^{1/p} < \infty \right\},
$$
which will be used mostly with $p=1$ and $2$. For any $n\in\N$ and
real number $l\in\R$, we also define $H_l^n$
as
\begin{equation}
  \label{def:Hkl}
H_l^n \,:=\, \left\{ f:\R^3\mapsto \R, \, \quad  \| f \|_{H^n_l} \,:=\,  \left(\sum_{|\alpha| \leq n} \| \partial_v^\alpha f \|_{L^2_l}^2\right)^{1/2} \,<\, \infty\, \right\}.  
\end{equation}

Before going further, let us make the following assumption on the exact solution to equation \eqref{1}.

\begin{hyp}[Basic Assumption on the initial data]\label{TAES}
Let us suppose  that the initial data $f_0$  to \eqref{1}
is nonnegative  such that
\begin{equation}
  \label{hyp:0}
\int_{\R^3} f_0(v) \left(\begin{array}{l} 1\\ v \\
                           |v|^2\end{array}\right) \dD v = \left(\begin{array}{l} 1\\ 0 \\
                           \ds 3 \end{array}\right)
\end{equation}
and
\begin{equation}
  \label{hyp:1}
f_0\in L^1_{\ell}\cap H^{n+2}_{k+l},
\end{equation}
where $(\ell,k,l)\in \R_+^3$ and $n\in\N$ satisfy $k>9/2$, $l>3$,
$n\ge 5$ and $\ell>19/2$ is sufficiently large  such that
$$
2\,(k+l)\,+\,3\,(n+2)\,+\,\f32\,<\,\frac{\ell(2 \,\ell^2\,-\,25\, \ell\,+\,57)}{2\,\ell^2-\,7\,\ell\,+\,21}.
$$ 
\end{hyp}
 
 The assumptions on the initial data  $f_0$ serve distinct but
 interconnected purposes in our analysis. First  moment Condition
 $f_0\in L^1_{\ell}$  will provide a control over the decay properties
 of $f_0$, enabling us to derive explicit convergence rates to the
 equilibrium. Such rates are crucial for establishing uniform-in-time
 bounds on the solution in weighted Sobolev spaces. Furthermore,
 regularity and decay conditions
 $$
 f_0\in H^{n+2}_{k+l}
 $$
 will  ensure sufficient smoothness for the propagation of regularity in the
 solution for the short-time.

To address these problems, our numerical approach is grounded in the
mathematical analysis of the equation, including its well-posedness,
smoothing properties, and the propagation of Fisher information.

%
%

\begin{prop}
\label{thm:fHnl}
Under the Assumption \ref{TAES} on $f_0$ and parameters
$(\ell,n,k,l)$, the equation \eqref{1} admits an unique global
solution $f\in L^\infty([0,\infty);H^{n+2}_{k+l})$ with initial data
$f_0$. We denote the bound by a constant $M$ depending only on $\ell$,
$n$, $k$ and $l$, that is,
\beq
\|f\|_{L^\infty([0,\infty);H^{n+2}_{k+l})}<M.
\eeq
\end{prop}

This proposition provides us with both a result of existence and
uniqueness of solution but also of regularity for the solution of the
Landau equation  \eqref{1}, which will subsequently allow us to obtain estimates
of errors between the approximate solutions and the solution to
\eqref{1}. The proof of Proposition \ref{thm:fHnl} will be given in
Section \ref{sec23}.

\subsection{Main results}
\label{sec13}

The objective of this study is on the one hand to construct an
approximation and on the other hand to provide error estimates by
considering a regular solution of the Landau equation. To do so, we
propose an approximation process based on the following two key steps:
on the one hand, we reduce the computational domain in velocity to a
bounded set $\cD_L$, where $L>0$, which ensures that the problem
becomes computable while preserving the essential dynamics of the
initial equation. On the other hand, we apply a Fourier-Galerkin
method preserving a finite number of modes $N>0$.

\subsubsection*{Truncated Landau equation and its variant formulation.}

Since the original Landau equation \eqref{1} is non-local, we first
introduce a truncated equation confined to a large box.  Let  $\psi^R\in C^\infty_0(\cB(0,R))$ be a radial 
bump function satisfying that
$$
\mathbf{1}_{\{|v|\le R/2\}}\le \psi^R\le
 \mathbf{1}_{\{|v|\le R\}},
 $$
 where $\mathbf{1}_\Omega$ denotes the characteristic function of a
 set $\Omega$. The truncated Landau equation is then defined by:
\ben\label{eq:trf}
\begin{cases}
  \ds\partial_t f^R\,=\, \cQ^R(f^R,f^R), \\[0.9em]
  \ds f^R|_{t=0}\,=\,f^R_0,
\end{cases}
\een
where $f^R_0:=f_0\psi^R$ while the   truncated collision operator $\cQ^R$ is given by
$$
\cQ^R(g,h)\,:=\,\cQ(g\,\psi^R,h\,\psi^R)\,\psi^R
.$$
 It is worth mentioning that we apply truncation both inside and outside the collision operator 
 $\cQ$. This choice has two important consequences. On the one hand,  by the standard ODE argument,
  the solution $f^R$ to \eqref{eq:trf} remains compactly supported in
  $\cB(0,R)$. On the other hand,  the formulation allows us to
  reinterpret the operator as a periodic Landau collision
  operator. This justifies and enables us to employee the Fourier
  spectral method used in the most literature. Furthermore, while the
  truncated Landau equation no longer preserves the conservation laws
  and entropy structure of the original equation, it retains two
  essential properties since solutions remain nonnegative and the dynamics coincide with the original equation inside the smaller ball
 $\cB({0,R/2})$. In other words, we can still capture  the original solution's behavior through the truncated equation if $R\gg1$.
  
We have following theorem which asserts the error estimate between \eqref{1} and \eqref{eq:trf}:

%
%

\begin{thm}
  \label{thm:truncationresult} 
Under the Assumption \ref{TAES} on $f_0$ and parameters $n$, $k$, $l$, consider $f$ the solution to \eqref{1} and the global regularity bound $M$ given in Proposition
\ref{thm:fHnl}. Then, for any $T>0$, there exist two  constants
$\tilde\kappa := \kappa_{n,\mathsf{p}}(M,n,k,l) > 0$, $\tilde
C:=\tilde C(M,n,k,l) >0$ and a
nondecreasing function $\tilde\RR$ given by
\begin{equation}
  \label{def:Rtilde}
\tilde\RR(t) \,:=\, \tilde C \,\exp\left(\tilde\kappa\,t\right),
\end{equation}
such that for
all $R > \tilde\RR(T)$, the truncated
equation \eqref{eq:trf} admits a unique, nonnegative solution
$f^R$  satisfying 
$$
{\rm Supp}(f^R(t,\cdot))\subset\cB(0,R), \qquad  \forall  \,t \in [0,T].
$$
Furthermore,  we have the following error estimates
\begin{equation}
  \label{errorbetweenfrf}
\|f^R(t) - f(t)\|_{H^n_k} \,\leq\, \frac{\tilde{\CC}}{R^{l}}\, \exp\left(\tilde{\kappa}\, t\right) \,\leq\, 1, \qquad \forall\, t\, \in [0,T],
\end{equation}
where the constants $\tilde{\CC}$ and $\tilde{\kappa}$ depend only on $M$, $n$,
$k$, $l$.
\end{thm}

As a direct consequence,  using the fact that $f^R$ is compactly
supported in $\cB(0,R)$, we may reformulate the truncated equation
\eqref{eq:trf}  as a  equation with so-called {\it periodic collision
  operator}.  Indeed, consider the torus $\cD_L=[-L,L]^3$  with $L\,=\,2R$ and
let $g$, $h\in H^1_{\#}(\cD_L)$, where $ H^1_{\#}(\cD_L)$ denotes the
$H^1$ Sobolev space  in the periodic domain $\cD_L$, the periodic Landau operator is defined as
\ben\label{defQp}
\cQ_\#(g,h)(v) \,:=\,\na\cdot\int_{[-L,L]^3}a^L(v-v_*)\left(g(v_*)\na h(v)\,-\,\na g(v_*)h(v)\right)\,\dD v_*,
\een
where $a^L(z)\,:=\,a|_{[-L,L]^3}(z)\,\mathbf{1}_{\{|z|<L\}}(z)$.
 Moreover, setting $\br{\cdot,\cdot}_\#$  the inner product for the
$L^2$ functions in the torus,   for each function $g$, $h$ and  $\varphi\in H^1_{\#}(\cD_L)$,   it holds that
  $$
 \br{\cQ_\#(g,h),\,\varphi}_{\#}\,=\,-\iint_{[-L,L]^6} a^L(v-v_*)\,\left(g(v_*)\na h(v)\,-\,\na g(v_*)h(v)\right)\,\na \varphi(v)\,\dD v_*\dD v.
 $$

As a consequence considering  $g$ and $h$ two functions compactly
supported in $\cB(0,L/2)$, we define two periodic function
$g_\#:=g\,\mathbf{1}_{\mathcal{D}_L}$ and
$h_\#:=h\,\mathbf{1}_{\mathcal{D}_L}$ in the torus
$\mathcal{D}_L$.  Hence  we further  have  that
$$
\cQ_\#(g_\#,h_\#)(v) \,=\,\cQ(g,h)(v), \qquad \forall\,v
\in\mathcal{D}_L.
$$
From these, considering a solution  $f^R$  to the truncated Landau
equation \eqref{eq:trf},   we set
$f^R_\#(t,v)\,:=\,f^R(t,v)\mathbf{1}_{\{v\in\mathcal{D}_L\}}(v)$ with
$R=L/2$, hence we rewrite \eqref{eq:trf} as follows:
 \ben\label{eq:trfp}
\begin{cases}
  \ds\partial_t  f^R_\#\,=\,\cQ_\#(f^R_\#\,\psi^R,\,f^R_\#\,\psi^R)\,\psi^R,
  \\[0.9em]
  \ds f^R_\#|_{t=0}\,=\,f^R_0\,\mathbf{1}_{\{v\in\mathcal{D}_L\}}.
\end{cases}
\een

\begin{rmk}
  The key difference between  $f^R_\#$ and $f^R$ lies in their domains
  of definition since $f^R_\#$ is defined on the bounded domain with
  periodic boundary conditions, whereas $f^R$ remains defined on the
  entire space $\R^3$. However, these functions are identical on $\cD_L$.
\end{rmk}

This approach successfully reduces the original problem in $\R^3$ to the periodic domain through equation \eqref{eq:trf}, while maintaining rigorous error control. This represents a significant improvement over conventional spectral methods, as evidenced by our novel error estimates.

\subsubsection*{Fourier spectral method.}
To give a numerical algorithm to solve \eqref{eq:trfp}, we apply a
Fourier-Galerkin method. To this aim we introduce 
$$
\mP_N\,:=\,{\rm Span}\left\{v\mapsto e^{-i\f{\pi}{L}k\cdot v}\  ;\ k\in\cJ_N\right\}
\,\subset \,L^2(\cD_L),
$$
with $\cJ_N:=\llbracket -N, N-1 \rrbracket^3$ and  $\cP_N$  the 
orthogonal projection from $L^2(\cD_L)$  onto the subspace $\mP_N$. More precisely, for any periodic function $f_\#\in L^2(\cD_L)$, we have  
\ben\label{defcPN}
\cP_N f_\#(v)\,=\,\frac{1}{(2L)^{3}}\,\sum_{k\in\cJ_N}\hat{f}_\#(k)\,\exp\left(i\f{\pi}{L}k\cdot v\right),
\een
where
$$
\hat{f}_\#(k):=\int_{\cD_L}f_\#(v)\,\exp\left(-i\f{\pi}{L}k\cdot v\right)\dD v.
$$
The truncated system takes the form:
\ben\label{eq:fRn}
\begin{cases}
  \ds\partial_t f_\#^{R,N}\,=\,\cQ^R_N\left(f_\#^{R,N},\,f^{R,N}_\#\right),
  \\[0.9em]
 \ds f_\#^{R,N}|_{t=0}\,=\,\cP_N\left(f^R_0\,\mathbf{1}_{\mathcal{D}_L}\right),
\end{cases}
\een
where $R:=L/2$, while the {\it truncated periodic collision operator in Fourier mode} $\cQ_N^R$ is defined through 
\ben
\label{defQNR}
\cQ_N^R(g,h)\,:=\,\cP_N\left(\cP_N\left(\cQ_\#\left(G_N,H_N \right)\right)\, \psi^R\right),
\een
with $G_N:=\cP_N\left(g\,\psi^R\right)$ and
$H_N:=\cP_N\left(h\,\psi^R\right)$.

It is worth mentioning that the operator $\cQ_N^R$ involves two key projections. On the one hand,
the first projection $\cP_N$  restricts the solution   $f_\#^{R,N}$ to
\eqref{eq:fRn}  to a finite number of Fourier modes. On the other
hand, the second projection $\cP_N$  ensures    $\cQ_N^R$ satisfies a   coercivity  estimate, which is crucial for later error analysis.

  Unlike the original truncated equation  \eqref{eq:trfp}, the solution $f^{R,N}_\#$ is constructed via Fourier series. As a result, we no longer retain guaranteed control over the nonnegativity of the solution.   Now we give an error estimate between  $f^{R,N}_\#$ and $f^R_\#$.

%
%

  \begin{thm}
    \label{thm:fourierresult}
 Under the Assumption \ref{TAES} on $f_0$ and parameters $n$, $k$,
 $l$.  For
 any $T>0$, there exist  positive constants $\CC$, $C_1$, and $\kappa$ depending only on
 $n$, $k$, $l$ and $M$ (the global regularity bound  given in Proposition
\ref{thm:fHnl})  such that if the truncation radius $R>0$ and the
truncation parameter $N>0$ are sufficiently large such that
 $$
 \begin{cases}
   \ds R \,>\, \tilde\RR(T),
   \\
   \ds \frac{N}{R} \,>\, \tilde\NN(T,R) \,:=\, C_1\,\exp\left(\frac{\kappa\,
       R^{1/2}\,T}{n-4}\right),
\end{cases}
$$
where $\tilde\RR$ is given in Theorem \ref{thm:truncationresult}.  Then, the numerical scheme \eqref{eq:fRn} admits a unique solution
$f_\#^{R,N} \in L^\infty([0,T]; L^2(\cD_L))$ and   we have  for all $t \in [0,T]$,
\begin{equation}\label{error2}
\|f^{R,N}_\#(t) - f_\#^R(t)\|_{L^2(\cD_L)} \,\leq\, \CC \,\left(\frac{R}{N}\right)^{n-2} \,\exp\left(\kappa\, R^{1/2}\,T\right) \,\leq\, 1,
\end{equation}
 where $f^R_\#$ is the solution to \eqref{eq:trfp}.
\end{thm}
 
It is worth mentioning here that the  error estimate \eqref{error2} is
only valid provided that the number $N$ satisfies that $N>L$, where
$L$ is the size of the computational domain. Gathering the previous
results, we are now ready to state our main result providing an error
estimate between the numerical approximation given by a
Fourier-Galerkin method on a truncated velocity domain and the the
exact solution to the Landau equation. We first define a numerical
approximation $f^{R,N}$  on $\R^3$ as
\beno
f^{R,N}(t,v)\,:=\,
\left\{\begin{array}{ll}
      \ds   f^{R,N}_\#(t,v), & {\rm if }\, v\in \cD_L,\\[0.9em]
         \ds 0, &  {\rm if }\, v\in \R^3\setminus \cD_L,
\end{array}\right.
\eeno
where $f^{R,N}_\#$ is given by \eqref{eq:fRn}. Then we prove the
following result.

%
%

\begin{thm}
  \label{thm:mainresult}
  Under the Assumption \ref{TAES} on $f_0$ and parameters $n$, $k$,
 $l$. For any $T>0$, let us choose  $L>0$ be the truncation of the
 velocity domain  and $N$ be the number of  Fourier modes such that 
\beno
L\,\ge\, 2\, \tilde\RR(T) \quad\mbox{and}\quad
\frac{N}{L}\,\geq\, \tilde\NN(T,L/2),
\eeno
where $\tilde\RR$ and $\tilde\NN$ are defined in Theorems \ref{thm:truncationresult}  and \ref{thm:fourierresult}  while the
constants $ C_1$ and $\kappa$ depend only on $M$, $n$, $l$ and
$k$. Then, there exist constants   $\CC$ and $ \kappa_0$ depending
only on $M$, $n$, $k$ and $l$ such that for each $t\,\in\, [0,T]$, we have
\ben
\label{errorEstimate}
\|f^{R,N}(t)-f(t)\|_{L^2(\R^3)}\,\le\, \CC\,\left[ \f{\exp\left(\kappa_0 t\right)}{L^l}\,+\, \left(\f{L}{N}\right)^{n-2}\,\exp\left(\kappa L^{1/2}\, t\right)\right].
\een  
\end{thm}
\begin{proof}
Using the triangular inequality and thanks to the conditions imposed
on $L$ and $N$, we know that  $f^R$ is compactly supported in $\cD_L$
and  coincides with $f^R_\#$ in $\cD_L$, hence we have
\begin{eqnarray*}
  \|f(t)  - f^{R,N}(t)\|_{L^2(\R^3)} &\le&
                                        \|f(t)-f^R(t)\|_{L^2(\R^3)}\,+\,\|f^{R}(t)-f^{R,N}(t)\|_{L^2(\R^3)}
  \\
                                  &=&
                                      \|f(t)-f^R(t)\|_{L^2(\R^3)}\,+\,\|f^{R}_\#(t)-f^{R,N}(t)\|_{L^2(\cD_L)}.
\end{eqnarray*} 
Then the desired result follows by applying Theorems \ref{thm:truncationresult} and  \ref{thm:fourierresult}.
  \end{proof}

  The main purpose of the next sections is on the one hand to prove error estimates
  between the solution $f$ of the Landau equation \eqref{1}  and 
  the solution of the truncated problem $f^R$, given by \eqref{eq:trf},
  corresponding to  Theorem \ref{thm:truncationresult}. On the other
  hand, we show error estimates (Theorem \ref{thm:fourierresult}) between the distribution $f^R$  and
  the approximation $f_\#^{R,N} $ given by the Fourier-Galerkin
  approach \eqref{eq:fRn}.

In the following section, we first recall, we present fundamental
results on the collision operator and establishes key analytical
properties of the equation. Then in Section \ref{sec:3}, we
investigate the truncated system (Theorem \ref{thm:truncationresult}),
proving local existence and uniqueness, obtained via standard energy
estimates,  propagation of regularity and  preservation of
nonnegativity ensuring that solutions are  physically
meaningful. Finally, in Section \ref{sec:4}, we focus on the
periodic collision operator and prove Theorem \ref{thm:fourierresult}, 
which guarantees the stability of the numerical method. 
In Section \ref{sec:numerical}, we present numerical simulations based
on the theoretical framework developed in the previous sections, 
demonstrating the spectral accuracy for Maxwellian
moldecules. Finally, we performed numerical simulations for Coulomb
potential and investigate the trend to equilibrium.

\section{Analytic results on the Landau equation}
\label{sec:2}
\setcounter{equation}{0}
\setcounter{figure}{0}
\setcounter{table}{0}

In the following we write $a \lesssim b$ or $a \gtrsim b$ if there exists a uniform
constant $C > 0$ such that $a \leq C\,b$ or $a \geq C\,b$. The notation
$a \sim b$ means both $a \lesssim b$ and $b \lesssim a$
hold. Moreover, constants $C_{a_1,\ldots,a_n}$ or $C(a_1,\ldots,a_n)$ denotes a constant depending on parameters
$a_1$,$\ldots$, $a_n$.

In this section, we present essential analytic results for equation \eqref{1}, including:
variational bounds (both lower and upper) for the collision operator,
commutator estimates  and
a proof of the uniform-in-time propagation of Sobolev regularity for solutions to \eqref{1}.

\subsection{Preliminary results}
\label{sec21}
We   present elementary results about the collision operator  which   will be used frequently in the later. 
\begin{prop}
  \label{thm:comm} 
Let $f$, $g$, $h$  be smooth functions and $m$ be a weight function  $m>0$. Then it holds that
\begin{eqnarray}
  \label{eq:gfhm2}
  \br{\cQ(g,f),\,h \,m^2}\,-\,\br{\cQ(g,fm),\,h\,m}
  &=&-2\,\int_{\R^3}(a*g):\left(\na f\otimes \na m\right)\, h\,m\, \dD v
  \\
  && -\int_{\R^3} (a*g):(\na^2m)\, f\,h\,m\, \dD v,
\nonumber
\end{eqnarray}
and
\begin{eqnarray}
  \label{eq:gffm2}
  &&\ds\br{\cQ(g,f),f\,m^2}-\br{\cQ(g,fm),fm} \,=\\[0.9em]
  &&\nonumber\ds\int \left[\,(a*g):\left(\frac{\nabla m}{m}\otimes \frac{\nabla
     m}{m}\right) \,+\,(b*g)\cdot \frac{\nabla m}{m}\,\right]\,f^2\,m^2\,\dD v\,,
\end{eqnarray}
where $b_i(z)\,:=\,\partial_j a_{i j}(z)$.
\end{prop}
\begin{proof}
The identity follows from the expansion:
 \beq
\cQ(g,f\,m)  
\,=\,\cQ(g,f)\,m\,+\,
2\,(a*g):\left(\nabla f\otimes \nabla m\right)\,+\,(a*g):(\nabla^2m)\,f.
\eeq
Hence, direct computations yield that
\beq
  \br{\cQ(g,f),\,h\,m^2}-\br{\cQ(g,f\,m),\,h\,m} \,=\; \br{\cQ(g,f)\,m\,-\,\cQ(g,fm),\,h\,m}
                                                 \,=\, 2 \, \cI_1 \,-\,
                                                     \cI_2,
\label{c:0}
                                                     \eeq
with
$$
\begin{cases}
 \ds\cI_1 \,:=\,
                     -\int(a*g):(\na f\otimes \na m)\, h\,m\, \dD v\,
                     \\[0.9em]
                     \ds\cI_2 \,:= \,\int (a*g):(\na^2 m)\, f\,h\,m\, \dD
                     v,
                   \end{cases}
 $$
                   which establishes the first identity \eqref{eq:gfhm2}.
                   \\
                   Furthermore, through integration by parts on
                   $\cI_1$ and using the definition of $b$, we get that
\begin{eqnarray*}
\cI_1 &=&\int \na\cdot\left( \left(a*g\right)\,\nabla m\, h\,m\right)\,
f\, \dD v
\\
&=& \int (a*g):\left(\na^2m\right)\, f \,h \,m\, \dD v \,+\,\int b*g\cdot \na m\,
f\,h\,m\, \dD v
\\
&+&\int (a*g):\left(\na m\otimes \left( m\na h+h\na m\right)\,\right)\, f\, \dD v
\\
&=& \cI_2 \,+\,\int \left[\, (a*g) :\left(\f{\na m}{m}\otimes \f{\na m}{m}\right)\,+\,b*g\cdot \f{\na
    m}{m}\right]\, f\,h\,m^2\, \dD v
\\
&+&\,\int (a*g):\left(\na m\otimes \na h\right)\, f\,m\, \dD v\,.
\end{eqnarray*}
Substituting this latter expression in \eqref{c:0}, it follows that,
\begin{eqnarray*}
 \br{Q(g,f),hm^2}\,-\,\br{Q(g,fm),hm} &=&  \int\left[ (a*g): \left(\frac{\nabla m}{m}\otimes
   \frac{\nabla m}{m}\right) \,+\, b*g \cdot \frac{\nabla
     m}{m}\,\right] fhm^2\,\dD v\\[0.9em]
 &+&\int \left[  (a*g):\left(\frac{\nabla m}{m}\otimes \rr{f\na h-h\na
       f}\right)\right]\,m^2\,\dD v.
 \end{eqnarray*}
Setting $f = h$ causes the last term to vanish, yielding \eqref{eq:gffm2}.
 \end{proof}

 \begin{prop}
   \label{thm:pizpq}
For each $p$, $q$ and $z\in\R^3$,  it holds that
   \ben
\label{Pizpq}
\begin{cases}
\ds |z|^2 \,\Pi(z) :\left( p \otimes q\right)  \,=\, (z \times p) \cdot (z \times q),
\\[0.9em]
\ds |\Pi(z)\,p| \,=\, \frac{|z \times p|}{|z|}.
\end{cases}
\een
Moreover, for each $v$, $v_*\in \R^3$, we have
\ben
\label{eq:avv}
\begin{cases}
\ds  a(v-v_*):\left(v\otimes v\right) \,=\,\frac{|v\times v_*|^2}{|v-v_*|^{3}},
  \\[0.9em]
\ds |a(v-v_*)\,v|\,=\,\frac{|v\times v_*|}{|v-v_*|^{2}}.
\end{cases}
\een
\end{prop}
\begin{proof} For the projection matrix $\Pi$ defined in \eqref{eq:aij} and for all $z$, $p$, $q \in \R^3$, we have  
\begin{eqnarray*}
 |z|^2\, \Pi(z) :\left( p \otimes q\right)  &=& \sum_{i=1}^3 |z|^2\, p_i
                                                q_i \,-\, \sum_{i,j=1}^3
                                                z_i \,z_j\, p_i\, q_j
  \\
                                            &= & \sum_{1 \le i < j \le
                                                 3} \left( z_j^2 \,p_i\, q_i \,+\,
                                                 z_i^2 \,p_j \,q_j \,-\, z_i
                                                 \,z_j \,p_i \,q_j \,-\, z_i \,z_j
                                                 p_j \,q_i\right)
                                                 \\
                                                 &=& (z \times p)
                                                     \cdot (z \times
                                                     q),
\end{eqnarray*}
which yields the first equality  of \eqref{Pizpq}.  Then using that
$\Pi$ is a symmetric projection operator, we have  $\Pi=\Pi^2$ and
taking $p=q$ in the latter equality, it gives
$$
|z \times p|^2 \,=\,  |z|^2 \, \Pi(z) :\left( p \otimes p\right) \,=\, |z|^2 \,
p^T\Pi^2(z)\,p\,=\, |z|^2 \,  |\Pi(z)\,p|^2,
$$
leading to the second equality of \eqref{Pizpq}.
\\
Finally \eqref{eq:avv} directly follows from the definition of $a$ in
\eqref{eq:aij} and \eqref{Pizpq} using that $|(v-v_*)\times
v|=|v\times v_*|$.
\end{proof}
\subsection{Commutator estimates and Sobolev bounds for the collision
  operator}
\label{sec22}
To establish bounds for the collision operator in weighted Sobolev
spaces, we introduce the weight function $m$,  typically taken to be
the polynomial weight $\wei^k$. We first generalize the definition of  the
space $H_l^r$ given in \eqref{def:Hkl} to any  $(r,l)\in\R^2$, 
$$
H_l^r := \left\{ f:\R^3\mapsto \R, \, \quad \|f\|_{H_l^r} = \left\| \langle D \rangle^r \langle \cdot \rangle^l f \right\|_{L^2} < +\infty \right\},
$$
where $\langle D \rangle^r$ is the Fourier multiplier with symbol
$\langle \xi \rangle^r$. Our first result will be the following lemma:
\begin{lem}
  \label{thm:commgfhm2}
  For all smooth functions $f$, $g$ and $h$,  the following inequality
  holds  for any $k > 2$,
 \begin{equation}
  \label{eq:commgfhm2} 
\left|\br{\cQ(g,f),h\wei^{2k}}-\br{\cQ(g,f\wei^k),h\wei^k}\right|
\,\lesssim_k\, \left(\|g\|_{L^1_2}+\|g\|_{L^{2}_2}\right)\,
\|f\|_{H^1_{k_1}}\, \|h\|_{L^2_{k_2}},
\end{equation}
where $(k_1,k_2)\in\R^2$ is such that  $k_1 + k_2 \,=\, 2\,k - 3$.
\\
Moreover, we have
\begin{equation}
\left|\br{\cQ(g,f),f\wei^{2k}}-\br{\cQ(g,f\wei^k),f\wei^k}\right|
\,\ls_k\, \left(\|g\|_{L^1_2}\,+\,\|g\|_{L^{2}_2}\right)\,
\|f\|_{H^{{1/4}}_{k-3/2}}^2.
\label{eq:commgffvk}
\end{equation}
\end{lem}
\begin{proof} 
  Let us set  $m(v) \,=\, \wei^k$ with $k > 2$. On the one hand, for
  any $v$ and $v_*\in\R^3$, we have
  $$
  \f{\na m(v)}{m(v)}\,=\,\frac{k\,v}{\wei^{2}},
  $$
  hence applying \eqref{eq:avv}  of Proposition \ref{thm:pizpq},  we
  get that
  \begin{equation}
  \label{a:1}
  |a(v-v_*)\na
  m(v) |\,=\,k\, |a(v-v_*)\,v|\wei^{k-2} \,=\, k\wei^{k-2}\,\frac{|v\times v_*|}{|v-v_*|^{2}}.
  \end{equation}
  On the other hand, since
  $$
  \f{\na^2  m(v)}{m(v)}\,=\,\frac{k}{\wei^{2}}\, \left({\rm Id}_3
    \,+\,(k-2)\, \frac{v \otimes v}{\wei^{2}}\right),
  $$
we have 
$$
\left|a(v-v_*):\left(\na^2 m(v)\right)\right| \,\le\, \frac{k\,m(v)}{\wei^{2}} \,
\left(\, \left| a(v-v_*): {\rm Id}_3\right| \,+\,(k-2)\,\wei^{-2}\, |a(v-v_*):\left(v\otimes v\right)|\,\right).
 $$
Using that $|a(v-v_*):{\rm Id}_3|=2\,|v-v_*|^{-1}$  and applying \eqref{eq:avv}  in Proposition \ref{thm:pizpq}, it holds that
\begin{equation}
 \label{a:2}
\left|a(v-v_*):\left(\na^2 m(v)\right)\right| 
\, \ls\, k\,\wei^{k-2} \; \left(\,\frac{1}{|v-v_*|} \,+\,
  \frac{k-2}{\wei^{2}}\, \frac{|v\times v_*|^2}{|v-v_*|^{3}} \,\right).
\end{equation}
Therefore, applying \eqref{eq:gfhm2} in Proposition \ref{thm:comm} with $m(v)=\wei^{k}$ and using
\eqref{a:1} -\eqref{a:2}, it yields that 
\begin{eqnarray*}
&& \left|\br{\cQ(g,f),\,h\wei^{2k}}\,-\,\br{\cQ(g,f\wei^k),\,h\wei^k}\right|  
\,\ls_k\,  \iint \frac{ |v\times v_*|}{|v-v_*|^{2}}\, |g_*\,\na f\,
        h|\,\wei^{2k-2}\,\dD v\,\dD v_*
  \\[0.9em]
  &&\qquad +\,\iint\frac{1}{ |v-v_*|^{3}}\, \left( |v-v_*|^2\,+\, \frac{|v\times
      v_*|^2}{\wei^{2}} \right) \,|g_*\,f\,h|\,\wei^{2k-2}\,\dD v\,\dD v_*.
\end{eqnarray*}
Moreover, since $|v\times v_*|\,=\,|v\times (v-v_*)|\,=\,|(v-v_*)\times v_*|$, we have:
\begin{equation}
  \label{c:22}
\begin{cases}
  \ds\frac{|v\times v_*|}{ |v-v_*|^{2}} \,\le\,\frac{|v_*|}{ |v-v_*|},
  \\[1.em]
\ds\frac{|v\times v_*|^2}{|v-v_*|^{3} } \,\le\, \frac{|v|^2}{
  |v-v_*|} \,\le\,\frac{\wei^2}{|v-v_*|}.
\end{cases}
 \end{equation}
Thus, we have:
\begin{eqnarray*}
  && \left|\br{\cQ(g,f),h\wei^{2k}}-\br{\cQ(g,f\wei^k),h\wei^k}\right|
  \\
  && \qquad \ls\,  \iint \frac{|v_*|}{|v-v_*|}\, |g_*\na f
        h|\wei^{2k-2}\,\dD v\,\dD v_* \,+\,\iint
        \frac{1}{|v-v_*|}\,|g_*fh|\wei^{2k-2}\,\dD v\,\dD v_*.
           \end{eqnarray*}
For the first term on the right hand side, we apply Theorem
\ref{Kalpha}-(3), with $\alpha=1$, $p=2>{3}/{(3-\alpha)}$, and
$\ell_0=1$, $\ell_1=-1$, $\ell_2=0$ and  $n_0=1$, $n_1=-1$, $n_2=0$ to
get that
\begin{eqnarray*}
\ds\iint \frac{|v_*|}{|v-v_*|}\, |g_*\,\na f\, h|\wei^{2k-2}\, \dD v\dD
  v_*&=&\ds\iint \frac{\weis |g_*|}{|v-v_*|}\, \rr{|\na f|\wei^{k_1+1}}
         \rr{|h|\wei^{k_2}}\dD v\dD v_*\\[0.9em]
&\ls& \ds\left(\|g\|_{L^1_2}+\|g\|_{L^{2}_2}\right)\,\|\na f\|_{L^2_{k_1}}\,\|h\|_{L^2_{k_2}},
\end{eqnarray*}
for any $k_1+k_2=2k-3$. Similarly, for the second term, applying Theorem \ref{Kalpha}-(3) with the same parameters, we have
\begin{eqnarray*} 
\iint \frac{1}{|v-v_*|}\,|g_*\,f\,h|\,\wei^{2k-2}\,\dD v\,\dD
  v_*&=&\iint\frac{|g_*|}{ |v-v_*| }\,\rr{| f|\wei^{k_1+1}}
         \,\rr{|h|\wei^{k_2}}\dD v\dD v_*\\[0.9em]
&\ls &\left(\Lab{g}{1}{1}+\|g\|_{L^{2}_{1}}\right)\,\|f\|_{L^2_{k_1}}\,\|h\|_{L^2_{k_2}},
\end{eqnarray*}
for any $k_1+k_2=2k-3$. Combining the latter results and using that for $k>2+3/2$,  $\|g\|_{L^1_2}\ls_k \|g\|_{L^2_k}$, we conclude that
$$
\left|\br{Q(g,f),h\wei^{2k}}\,-\,\br{Q(g,f\wei^k),h\wei^k}\right| \,\lesssim_k\, \|g\|_{L^2_k}\,\|f\|_{H^1_{k_1}}\,\|h\|_{L^2_{k_2}}.
$$
 This completes the proof of \eqref{eq:commgfhm2}. Now let us turn to
 the proof of \eqref{eq:commgffvk}  by applying \eqref{eq:gffm2} in
 Proposition \ref{thm:comm},  it yields that
 $$
\left|\br{Q(g,f),\,f\wei^{2k}}\,-\,\br{Q(g,f\wei^k),\,f\wei^k}\right|
\, \ls_k \, \cI_1 \,+\, \cI_2,
$$
with
$$
\begin{cases}
\ds\cI_1\,:=\, \iint \frac{|v\times
  v_*|^2}{|v-v_*|^{3}}\,g_*f^2\,\wei^{2k-4}\,\dD v\,\dD v_*,
\\[0.9em]
\ds\cI_2  \,:=\, \iint \frac{1}{|v-v_*|^{2}}\,g_*\,f^2\,\wei^{2k-1}\,\dD
v\,\dD v_*.
\end{cases}
$$
We  first treat $\cI_1$ using \eqref{c:22}, 
\beq
\cI_1\,\ls\, \iint \frac{|v_*|^2}{|v-v_*|} \,g_*\, f^2\,\wei^{2k-4}\,\dD
v\,\dD v_*,
\eeq
hence applying Theorem \ref{Kalpha}-(3) with $\alpha=1$,
$\ell_0=\ell_1=\ell_2=0$, $n_0=n_1=n_2=0$ and $p=2$ we obtain
\beq
|\cI_1|\,\ls\, \iint \frac{1}{|v-v_*|}\,\rr{|g_*|\weis^2}
\,\rr{f\wei^{k-2}}^2\dD v\,\dD v_*\,\ls\, \left(\|g\|_{L^1_2}\,+\,\|g\|_{L^2_2}\right)\,\|f\|_{L^2_{k-2}}^2.
\eeq
Next,  for $\cI_2$. we apply Theorem \ref{Kalpha}-(2) with $\alpha=2$,
$\sigma_1=\sigma_2=1/4$, so that
${3}/{(3-\alpha+\sigma_1+\sigma_2)}=2$, and $\ell_0=2$,
$\ell_1=\ell_2=-1$, $n_0=2$, $n_1=n_2=-1$ and  find that
\beq
|\cI_2|\,\ls\,\iint\frac{|g_*|}{ |v-v_*|^{2}}\,\rr{f\wei^{k-1/2}}^2
\;\dD v\,\dD v_*\,\ls\, \|g\|_{L^1_{2}} \, \|f\|_{L^2_{k-3/2}}^2\, +\,
\Lab{g}{2}{2}\,\|f\|_{H^{{1/4}}_{k-3/2}}^2\,.
\eeq
Combining these latter estimates and since $k-2 < k-3/2$,
$\|f\|_{L^2_{k-2}}\le \|f\|_{H^{{1/4}}_{k-3/2}}$,  hence we  get
the estimate \eqref{eq:commgffvk}.
\end{proof}

Now we are in a position to   derive the   the upper bounds for
$\br{\cQ(g,f), h\wei^{2k}}$. To this aim we introduce an anisotropic Sobolev space for collision operator:  
\ben
\label{eq:defh1s}
\mathsf{H}^1_{\mathsf{A}}\,:=\,\left\{ f:\R^3\mapsto \R, \quad \|f\|_{\mathsf{H}^1_{\mathsf{A}}}^2\;=\;\Lab{(-\Delta_{\S^2})^{\f{1}{2}}f}{2}{-3/2}^2+\Hab{f}{1}{-3/2}^2<\infty  \right\},
\een
where $-\Delta_{\mathbb{S}^2} := \sum_{1\leq i<j\leq 3} (v_i\partial_j - v_j\partial_i)^2$ is the Laplace-Beltrami operator on $\mathbb{S}^2$. 

\begin{lem}
  \label{thm:upperboundQgfh}
For any smooth functions $f$ ,$g$ and $h$, it holds for $k>2$, 
\ben
\label{eq:Qgfhv2k2}
\left|\br{\cQ(g,f),\,h\wei^{2k}}\right|\,\ls\, \left(\|g\|_{L^1_2}\,+\,\|g\|_{L^{2}_2}\right)\,\|f\|_{H^{2}_{k-1}}\,\|h\|_{L^2_k}\,.
 \een
 Moreover, for $k>3+3/2$, we also have
 \ben
 \label{eq:Qgfhv2k4}
\left|\br{\cQ(g,f),\,h\wei^{2k}}\right| \,\ls_k\,
\Lab{g}{2}{k}\,(\|f\wei^k\|_{\H1A} \, \|h\wei^k\|_{\H1A}\,+\,\|f\|_{H^1_{k-2}}\,\|h\|_{L^2_k}),
\een
and
\ben
\label{eq:Qgfhv2k3}
\left|\br{\cQ(g,f),\,h\wei^{2k}}\right| \,\ls_k\,  \Lab{g}{2}{k}\,\|f\wei^k\|_{\H1A}\,\left(\|h\wei^k\|_{\H1A}\,+\,\|h\|_{L^2_k}\right).
\een
\end{lem}
\begin{proof}
  Let us take $k>2$ and  apply \eqref{eq:commgfhm2} of Lemma
  \ref{thm:commgfhm2} with $k_1=k-3$ and $k_2=k$, we get  that
$$
\left|\br{\cQ(g,f),h\wei^{2k}}\right|\,\ls\,
\left|\br{\cQ(g,f\wei^k),h\wei^{k}}\right|\,+\,\left(\|g\|_{L^1_2}+\|g\|_{L^{2}_2}\right)
\,\Hab{f}{1}{k-3}\,\Lab{h}{2}{k}.
$$
Then we apply Proposition \ref{UpperboundofQ}-$(i)$
  with $a_1=2,b_1=0$, $\omega_1=-1$ and $\omega_2=0$ to  estimate the
first term on the right hand side, it yields that
$$
\left|\br{\cQ(g,f\wei^k),\,h\wei^{k}}\right|\,\ls\, \left(\|g\|_{L^1_2}\,+\,\|g\|_{L^{2}_2}\right)\,\|f\|_{H^{2}_{k-1}}\,\|h\|_{L^2_k},
$$
which yields \eqref{eq:Qgfhv2k2}.
\\
For $k>3+3/2$, we proceed in the same manner using
\eqref{eq:commgfhm2} of Lemma \ref{thm:commgfhm2} with $k_1=k-3$ and $k_2=k$, 
$$
\left|\br{\cQ(g,f),\,h\wei^{2k}}\right| \,\ls\, \left|\br{\cQ(g,f\wei^k),h\wei^{k}}\right|\,+\,\Lab{g}{2}{k}\Hab{f}{1}{k-3}\,\Lab{h}{2}{k},
$$
whereas Proposition \ref{UpperboundofQ}-$(ii)$ with $a_2=1$ ,$b_2=0$,
$\omega_1=\omega_2=-3/2$, $\omega_3=-2$, and $\omega_4=0$ gives 
$$
\left|\br{\cQ(g,f\wei^k),\,h\wei^{k}}\right|\,\ls\,\left(\Lab{g}{1}{3}\,+\,\|g\|_{L^{2}_{3}}\right)\,\left(\|f\wei^k\|_{\H1A}\,\|h\wei^k\|_{\H1A}\,+\,\|f\|_{H^1_{k-2}}\,\|h\|_{L^2_k}\right).
$$
Gathering these latter results, with $\Lab{g}{1}{3}+\|g\|_{L^{2}_{3}}\ls \|g\|_{L^2_k}$ when $k>3+3/2$, we first get \eqref{eq:Qgfhv2k4}.  Then, since  $\|f\|_{H^1_{k-3/2}}\ls
\|f\wei^k\|_{\H1A}$,  by \eqref{eq:h1s}  in Lemma \ref{thm:fs2}, it
also yields \eqref{eq:Qgfhv2k3}.
\end{proof}

Thus we have following corollary:

\begin{cor}
  \label{thm:Qgfvk}
For any function $f$, $g$ and $k>2+3/2$, we have
\ben
\label{eq:Qgfvk}
\|\cQ(g,f)\|_{L^2_k}\ls_k \Lab{g}{2}{k}\Hab{f}{2}{k-1}.
\een
Moreover,
\ben
\label{eq:Qgffv2k}
\left|\br{\cQ(g,f),f\wei^{2k}}\right|\ls_k \Lab{g}{2}{k}(\|f\wei^k\|_{\H1A}^2+\|f\|_{L^2_{k}}^2).
\een
\end{cor}
\begin{proof}
On the one hand, thanks to \eqref{eq:Qgfhv2k2}, with $\|g\|_{L^1_2}+\|g\|_{L^2_2}\ls \|g\|_{L^2_k}$ when $k>2+3/2$, it holds that
\beq
\bigg|\br{\cQ(g,f)\wei^k,h}\bigg|=\bigg|\br{\cQ(g,f)\wei^k,h\wei^{-k}\wei^k}\bigg|\ls \Lab{g}{2}{k}\Hab{f}{2}{k-1}\|h\|_{L^2},
\eeq
hence taking $h= \cQ(g,f)\wei^k$, it  gives \eqref{eq:Qgfvk}.
\\
Furthermore, using \eqref{eq:Qgfhv2k3} with $h=f$ and the Young
inequality, it gives \eqref{eq:Qgffv2k}
\end{proof}

Next we shall give the  lower bound  for $\br{\cQ(g,f),f\wei^{2k}}$
when $g\ge0$.
\begin{lem}
\label{thm:lowerboundQgff}
Suppose that $f$ is a smooth function and $g\ge0$ satisfies
$\|g\|_{L^1}>\delta$ and
$$
\|g\|_{L_2^1}\,+\,\int_{\R^3}g \log(1+g) \dD v \,<\,\lambda.
$$
Then for $k>2$,  it holds that
$$
\br{\cQ(g,f),f\wei^{2k}} \,\le\, -K\,\|f\wei^k\|_{\H1A}^2\,+\,C_2\,\rr{1\,+\,\|g\|_{L^2_{3}}^{4}}\,\|f\|_{L^2_{k-3/2}}^2\,,
$$
where the constant $K$  depends only on $\delta$ and $\lambda$ while the
constant $C_2$  depends only on $\delta$, $\lambda$ and $k$.
\end{lem}
\begin{proof} We first observe that, for any function $h$, it holds 
\beq
\ba
&\br{\cQ(g,h),h}  \,=\, -\int_{\R^3}  (a*g):\left(\na h\otimes \na h\right) \,\dD v\,+\,4\,\pi\, \int_{\R^3}  g\,h^2\,\dD v.
\ea
\eeq
Thus, by 
    Theorem \ref{coe}, we get that
\beq
-\int_{\R^3}  (a*g):\na h\otimes \na h \,\dD v\,\le\, -\,C_1\,\|h\|_{\H1A}^2\,+\,C_2\,\|h\|_{L^2_{-\f32}}^2\,,
\eeq
where $C_1$ is a constant depending only on $\delta,\lambda$. By 
Holder's inequality and Sobolev embedding, we have 
\beq
\ba
\int_{\R^3} g\,h^2\,\dD
v\,\le\,\|g\|_{L^2_{3}}\|h\|^{1/2}_{L^2_{-3/2}}\|h\|_{L^6_{-3/2}}^{3/2}
\,\ls\, C(\varepsilon)\,\|g\|_{L^2_{3}}^{4}\,\|h\|_{L^2_{-3/2}}^2\,+\,\varepsilon\|h\|_{H^1_{-3/2}}^2,
\ea
\eeq
which implies that, taking $\varepsilon<C_1/2$ small enough, with
$\|h\|_{H^1_{-3/2}}^2\ls \|h\|_{\H1A}$ from \eqref{eq:h1s} of Lemma \ref{thm:fs2}, we get that
\beq
\br{\cQ(g,h),h}\le -\f{C_1}{2}\|h\|_{\H1A}+C_2(1+\|g\|_{L^2_{3}}^{4})\|h\|_{L^2_{-3/2}}^2.
\eeq
Taking $h=f\wei^k$ and applying  \eqref{eq:commgffvk} in Lemma
\ref{thm:commgfhm2} with the condition  $\|g\|_{L^1_2}<\lambda$, it yields that  
\begin{eqnarray*}
\ds\br{\cQ(g,f),f\wei^{2k}} &\le&\ds \br{\cQ(g,f\wei^k),f\wei^{k}}+C_3(\Lab{g}{1}{2}+\Lab{g}{2}{2})\|f\|_{H^{{1/4}}_{k-3/2}}^2\\
&\le& \ds
      -\f{C_1}{2}\|f\wei^k\|_{\H1A}^2\,+\,C_2\,\left(1+\|g\|_{L^2_{3}}^{4}\right)\,\|f\|_{L^2_{k-3/2}}^2\\
  &+&C_3\,\left(\lambda+\Lab{g}{2}{2}\right)\,\|f\|_{H^{{1}}_{k-3/2}}^{1/2}\,\|f\|_{L^{{2}}_{k-3/2}}^{3/2},
\end{eqnarray*}
where $C_3$ depends only on $k$. Finally, since
\beq
C_3(\lambda+\Lab{g}{2}{2})\|f\|_{H^{{1}}_{k-3/2}}^{1/2}\|f\|_{L^{{2}}_{k-3/2}}^{3/2}\le \f{C_1}{4}\|f\|_{H^{{1}}_{k-3/2}}^{2}+C_3'(\lambda+\Lab{g}{2}{2})^{4/3}\|f\|_{L^{{2}}_{k-3/2}}^{2},
\eeq
with $(\lambda+\Lab{g}{2}{2})^{4/3}\ls_{\lambda}1+ \Lab{g}{2}{2}^{4/3}\ls 1+\Lab{g}{2}{3}^4$, we derive the desired result.
\end{proof}

\begin{rmk}  In fact, if we   focus only on the Sobolev regularity, by the same manner, we can obtain that
\ben\label{eq:Qghh}
\br{\cQ(g,h),h}\ls -K\|\na h\|_{L^2_{-3/2}}^2+C_2(1+\|g\|_{L^2_3}^4)\|h\|_{L^2_{-3/2}}^2.
\een
\end{rmk}

\subsection{Proof of Proposition \ref{thm:fHnl}}
\label{sec23}

First we prove the global well-posedness. Since by Assumption \ref{TAES} the initial data satisfies $f_0 \in L^1_\ell \cap H^{n+2}_{k+l}\subset L^1_{7}\cap L^2_{3}$ as $\ell>19/2>7$ and $k+l>3$, we apply Theorem \ref{thm:coulombexist}-(1) with $\mathsf{r}=0,m=3$ to obtain the following global well-posedness result:

The solution $f$ to the Landau equation \eqref{1} associated with the initial data $f_0$ is the unique global solution satisfying that
\ben\label{eq:gw}
f\in C([0,\infty); L^2_3)\cap L^2_{loc}([0,\infty); H^1_{3/2}).
\een

Then, we want to prove that the equation
propagates regularity
$$
f\in L^\infty([0,\infty);H^{n+2}_{k+l}).
$$
Applying Theorem \ref{thm:coulombexist}-(3) with  $f_0\in  H^{n+2}_{k+l}$, there exists
$T>0$ and $M_1$ such that
\ben\label{eq:loc}
\sup_{t\in [0,T]}\|f(t)\|_{H^{n+2}_{k+l}}\,<\,M_1,
\een
hence it remains to prove that holds in the large time interval
$[T,+\infty)$. Unfortunately, Theorem \ref{thm:coulombexist}-(2) does not
provide directly a uniform bound on this quantity, hence we introduce a sequence of time intervals $(I_j)_{j\geq 1}$ with
$I_j\,:=\,(t_{j-1},\,t_{j+1}]$, with $t_j=j\,T$  for any $j\geq 1$ of
size $2T$ and regard $f|_{I_j}$ as the solution to \eqref{1}
associated with initial data $f(t_{j-1})$. Assuming  that  
\begin{equation}
\label{hyp:j}
 f|_{I_j}\in C\left([0,\infty); H^{-1/2}_{k+l+\f{3}{2}(n+2)+\f34}\right)\,\cap\, L^2_{loc}\left([0,\infty); H^{1/2}_{k+l+\f{3}{2}(n+2)+\f34-\f32}\right),
\end{equation}
we may apply Theorem \ref{thm:coulombexist}-(2) with 
$\mathsf{r}=-1/2,m=k+l+\f{3}{2}(n+2)+\f34,\nn=n+2$ 
 on the interval
$I_j$ to  get that for each $t\in [t_j, t_{j+1}]\subset \,I_j$, as $m-\f32\nn+\f32 \mathsf{r}=k+l$, it follows that
\begin{eqnarray*}
\|f(t)\|_{H^{n+2}_{k+l}} &\le&
\frac{C}{(t-t_{j-1})^{\f{n+2}{2}+\f14}}\,\mathbf{1}_{\{t-t_{j-1}\leq 1\}} \,+\,C_j(t)\,\mathbf{1}_{\{t-t_{j-1}>1\}}
  \,\leq\, \frac{C}{T^{\f{n+2}{2}+\f14}} \,+\,C_j(t)\,,
\end{eqnarray*}
for some constant $C$ and some locally uniformly bounded function
$C_j$ depending only on a upper bound of
$\|f(t_{j-1})\|_{H^{-1/2}_{k+l+\f{3}{2}(n+2)+\f34}}$. Therefore,  as soon as  $f|_{I_j}$ satisfies \eqref{hyp:j} and
the function $C_j$ is bounded independently of $j$, we will get that
for any $j \geq 1$
$$
\sup_{t\in [t_j,t_{j+1}]} \|f(t)\|_{H^{n+2}_{k+l}} \leq M_2,
$$
hence the results will follow with $M=\max(M_1,M_2)$. To this aim, we provide uniform-in-time bound on  moments of $f$
in $L^1$ and on the Fisher information
$$
\cI(f)\,:=\,\int_{\R^3} \f{|\na f|^2}{f} \dD v.
$$
On the one hand, since $f_0\in L^1_\ell$, we apply  Theorem
\ref{thm:coulombexist}-(5) together with the  Csiszár-Kullback-Pinsker
inequality, which yields 
$$
\|f\;-\;\mu\|_{L^1}\ls \sqrt{\cH(f\;|\;\mu)}\ls_{\beta}
(1+t)^{-\beta/2},\qquad{\rm with }\  \beta<\frac{2 \ell^2-25
  \ell+57}{9(\ell-2)}, 
$$
where $\mu$ is  the normalized Maxwellian.
On the other hand, applying Theorem \ref{thm:coulombexist}-(5) we get
that  $\|f\|_{L^1_\ell}\ls_{\ell} (1+t)$, hence it gives by
interpolation inequality that 
\begin{eqnarray*}
\ds \|f-\mu\|_{L^1_{\f{\ell\,\beta}{2+\beta}}} \;\le\,
\ds\|f-\mu\|_{L^1}^{\f{2}{2+\beta}}\, \|f-\mu\|_{L^1_\ell}^{\f{\beta}{2+\beta}}
\,\ls\,  ((1+t)^{-\beta/2})^{\f{2}{2+\beta}}(1+t)^{\f{\beta}{2+\beta}}\,=\,1.
\end{eqnarray*}
Thus, we get the following estimate
\begin{equation}
  \label{regu:1}
\|f\|_{L^1_{\f{\ell \,\beta}{2+\beta}}} \,\le\,
\|f-\mu\|_{L^1_{\f{\ell\,\beta}{2+\beta}}}
\,+\,\|\mu\|_{L^1_{\f{\ell\beta}{2+\beta}}}\,\ls\; 1.
\end{equation}
Then, the key idea is to prove regularity of $f$ using the monotonicity
property of Fisher information $\cI(f)$ already used in \cite[Lemma
1]{TV00},  from which we get that
$\cI(f)\ls_{\varepsilon}\|f\|_{H^2_{3/2+\varepsilon}} $ for
$\varepsilon>0$, hence as $n\ge 0, k+l>3/2$ we have $\cI(f_0)\ls \|f_0\|_{H^2_{k+l}}\le \|f_0\|_{H^{n+2}_{k+l}}<+\infty$. Therefore
from the monotonicity property of the Fisher information provided in Theorem \ref{thm:coulombexist}-(4), we have
\begin{equation}
  \label{regu:2}
\cI(f)(t) \,\leq\, \cI(f_0)\,<\,+\infty, \qquad \forall \,t\,\ge \, 0.
\end{equation}
Using the classical Sobolev embedding $L^{3/2}\hookrightarrow
H^{-1/2}$ and $\|f\|_{L^3}\ls \cI(f)$ it holds that
\begin{eqnarray*}
\|f\|_{H^{-1/2}_{k+l+\f{3}{2}(n+2)+\f34}} &\ls&
                                                \|f\|_{L^{3/2}_{k+l+\f{3}{2}(n+2)+\f34}}
  \\
                                          &\ls& \|f\|_{L^3}^{1/2}\|f\|_{L^{1}_{2(k+l)+3(n+2)+\f32}}^{1/2}
                                                \\&\ls& \cI^{1/2}(f)\,\|f\|_{L^{1}_{2(k+l)+3(n+2) +\f32}}^{1/2}.
\end{eqnarray*}
From Assumption \ref{TAES} we could find some $\beta$ satisfying $2(k+l)+3(n+2) +\f32\le \ell\f{\ell\,\beta}{2+\beta}$ and $\beta<\frac{2 \ell^2-25 \ell+57}{9(\ell-2)}$ at the same time.
Thus we know from \eqref{regu:1} that
$$
\|f\|_{L^{1}_{2(k+l)+3(n+2) +\f32}}\,\le\,
\|f\|_{L^1_{\ell\f{\beta}{2+\beta}}}\ls 1.
$$
Then from \eqref{regu:2}, we obtain that for all $t\ge 0$,
\beq
\|f(t)\|_{H^{-1/2}_{k+l+\f{3}{2}(n+2)+\f34}}\,\ls\, \cI^{1/2}(f)(t) \,\|f(t)\|_{L^{1}_{2(k+l)+3(n+2) +\f32}}^{1/2} \ls \cI^{1/2}(f_0)\,\ls\, 1,
\eeq
which gives the expected  uniform bound $t\ge 0$
\beq
\sup_{t\ge 0}\|f(t)\|_{H^{-1/2}_{k+l+\f{3}{2}(n+2)+\f34}}\,<\,M_0.
\eeq
Finally by applying Theorem \ref{thm:coulombexist}-(1) with $\mathsf{r}=1/2$,
$m=k+l+\f{3}{2}(n+2)+\f34$ we  obtain that
\eqref{hyp:j} holds, which concludes the proof. 

\section{Proof of Theorem \ref{thm:truncationresult}: Analysis of
  truncated Landau Equation}
\label{sec:3}
\setcounter{equation}{0}
\setcounter{figure}{0}
\setcounter{table}{0}

This section is devoted to the analysis of truncated Landau Equation \eqref{eq:trf}. In particular, we shall focus on the well-posedness, propagation of the regularity, the nonnegativity.

\subsection{Reformulation of \eqref{eq:trf} and   bounds for truncated
  operator}
\label{sec31}
We define the error function $g := f^R - f$ governed by:
\ben\label{eq:trg}
\begin{cases}
    \ds\partial_t g = \cL_{f}^R\,g+\cQ^R(g,g)+\cR(f), \\[0.9em]
    \ds g|_{t=0} = f_0(\psi^R - 1).
\end{cases}
\een
where  the linear operator  $\cL_{f}^R\,g$ and the residual operator $\cR(f)$ are defined by 
\ben
\label{eq:res}
\begin{cases}
  \ds\cL_{f}^R\, h\,:=\,\cQ^R(f,h)\,+\,\cQ^R(h,f),
  \\[0.9em]
\ds\cR(f)\,:=\,\cQ(f\psi^R,f\psi^R)\,\psi^R\,-\,\cQ(f,f).
\end{cases}
\een
 
For the linear operator, we have the following coercivity estimate: 
\begin{lem}
  \label{thm:Lhhv2k}
 Under the Assumption \ref{TAES} - \eqref{hyp:0} and 
  $$
\cH_0 := \int_{\R^3} f_0 \log (f_0) \dD v  < \infty,
  $$
we  consider  $f\in L^{\infty}([0,\infty),L^1_2)$  a solution to \eqref{1}. Then, for $R>3$, $k>3+3/2$, and $t\ge0$, the linear operator $\cL_{f(t)}^R$ satisfies the following estimate:
$$
\br{\cL_{f(t)}^R\, h,h\wei^{2k}}\,\le\, -K\,\|h\psi^R\wei^k\|_{\H1A}^2\,+\,C_1\left(1+\|f\|_{L^2_3}^4+\|f\|_{H^2_{k-1}}\right)\,\|h\|_{L^2_k}^2,
$$
where the constant $K$  depends only on $\cH_0$ while $C_1$ depends on
$\cH_0$ and $k$.
\end{lem}
\begin{proof}
  Since the distribution function $f\geq 0$  verifies the conservation laws and the monotonicity of
  the entropy, we have
  $$
  \int_{\R^3} f(t,v) \left(\begin{array}{l} 1\\ v \\
                           |v|^2\end{array}\right) \dD v = \left(\begin{array}{l} 1\\ 0 \\
                                                                   \ds 3 \end{array}\right)
                                                               $$
and 
$$
\cH(f(t)) \,:=\,\int_{\R^3} f(t,v)\log f(t,v) \dD v\,\le\, \cH_0\,,
$$
hence for $R>3$, we get that 
$$
\begin{cases}
\ds \|f\psi^R\|_{L^1}\,=\,\|f\|_{L^1}-\|f\,(1-\psi^R)\|_{L^1}\,>\,\f{1}{2}\,, 
\\[0.9em]
\ds \|f\psi^R\|_{L^1_2} \,\le\, \|f\|_{L^1_2}\,=\,3\,,
\\[0.9em]
\ds\int_{\R^3} f\psi^R\, \log ( f\psi^R) \dD v \,\ls\, \cH(t)+1 \,\leq \,\cH_0+1.
\end{cases}
$$
Here the last inequality is due to the fact
$$
\int_{\R^3} f\psi^R\, \log ( f\psi^R) \dD v\le \int_{\R^3} f\psi^R\, \log (1+ f\psi^R) \dD v\le \int_{\R^3} f\, \log (1+ f) \dD v\ls \cH(t)+\|f\|_{L^1_2}+1.
$$

Thus, by Lemma \ref{thm:lowerboundQgff}, we get that there exist
constants, $K$, depending only on $\cH_0$, and $C_2$,  depends only on
$\cH_0$ and $k$, such that
$$
\br{\cQ(f\psi^R,h\psi^R),h\psi^R\wei^{2k}} 
\,\le\, -K\|h\psi^R\wei^k\|_{\H1A}^2\,+\,C_2\left(1+\|f\|_{L^2_3}^4\right)\,\|h\|_{L^2_{k-3/2}}^2,
$$
with $k>3+3/2$. Thanks to \eqref{eq:Qgfhv2k2} of Lemma \ref{thm:upperboundQgfh}, we obtain that
\begin{eqnarray*}
\br{\cL_{f(t)}^R\, h,h\wei^{2k}} 
&=&\br{\cQ(f\psi^R,h\psi^R),h\psi^R\wei^{2k}}+\br{\cQ(h\psi^R,f\psi^R),h\psi^R\wei^{2k}}\\
&\le& -K\,\|h\psi^R\wei^k\|_{\H1A}^2\,+\,C_1\left(1+\|f\|_{L^2_3}^4+\|f\|_{H^2_{k-1}}\right)\,\|h\|_{L^2_k}^2,
\end{eqnarray*}
which implies the desired result.
\end{proof}

Next, we  give a upper bound for  the residual operator: 
\begin{lem}
  \label{thm:residual}
For $l>0$, $ k>2+3/2$ and any function $h\in H^2_{k+l}$, it holds
$$
\|\cR(h)\|_{L^2_k}\,\ls_{k,l}\,\frac{1}{R^{l}}\,\|h\|_{H^2_{k+l}}^2.
$$
\end{lem}
\begin{proof}
We emphasize that the smallness stems from the term $1-\psi^R$. Since
\begin{eqnarray*}
  \ds\cR(h) &=&\cQ(h\psi^R,h\psi^R)\,\psi^R\,-\,\cQ(h,h)
  \\[0.9em]
  &=& \ds\cQ\left(h\left(\psi^R-1\right),h\psi^R\right)\,\psi^R\,+\,\cQ\left(h,h\,\left(\psi^R-1\right)\right)\;\psi^R\;,+\,\cQ(h,h)\;\left(1-\psi^R\right),
\end{eqnarray*}
by \eqref{eq:Qgfvk} of Corollary \ref{thm:Qgfvk}, 
for $l>0$, $k>2+3/2$,   it holds that
$$
\|\cQ(h(\psi^R-1),h\psi^R)\psi^R\|_{L^2_k} \,\ls_k\, \|h(\psi^R-1)\|_{L^2_k}\,\|h\psi^R\|_{H^2_{k-1}}\,\ls\, \frac{1}{R^{l}}\,\|h\|_{L^2_{k+l}}\;\|h\|_{H^2_{k-1}}\,.
$$
Similarly, we have
$$
\begin{cases}
\ds \|\cQ(h,h(\psi^R-1))\psi^R\|_{L^2_k} \,\ls\, \|h\|_{L^2_{k}} \,
\|h(\psi^R-1)\|_{H^2_{k-1}}
\,\ls\,\frac{1}{R^{l}}\,\|h\|_{L^2_{k}}\,\|h\|_{H^2_{k-1+l}},
\\[0.9em]
\ds\|\cQ(h,h)(\psi^R-1)\|_{L^2_k}\,\ls\,
\frac{1}{R^{l}}\,\|\cQ(h,h)\|_{L^2_{k+l}} \,\ls\, \frac{1}{R^{l}}\,\|h\|_{L^2_{k}}\,\|h\|_{H^2_{k-1+l}}.
\end{cases}
$$
These imply  that
\beq
\|\cR(h)\|_{L^2_k}\,\ls\, \frac{1}{R^{l}}\,\|h\|_{L^2_{k+l}}\,\|h\|_{H^2_{l+k-1}}\,\le\, \frac{1}{R^{l}}\,\|h\|_{H^2_{l+k}}^2.
\eeq
\end{proof}

\subsection{Well-posedness and Propagation of Sobolev regularity}
\label{sec32}

In this subsection, we  give a proof for the first part of
Theorem \ref{thm:truncationresult}  on the error estimate.  Applying
Proposition \ref{thm:fHnl}, we know that  the solution $f$ to the equation \eqref{1} with the initial data $f_0$ satisfies that 
\beq
\|f\|_{ L^{\infty}([0,\infty),H^2_{k+l})}<M.
\eeq
Hence, applying the elementary energy method to \eqref{eq:trg},  it
yields that
$$
 \f{1}{2}\f{\dD}{\dD t}\|g\|_{L^2_k}^2\,=\,\br{\cL_{f(t)}^R\, g,\,g\wei^{2k}}\,+\,\br{\cQ^R(g,g),g\wei^{2k}}\,+\,\br{\cR(f(t)),g\wei^{2k}}\,.
$$
Using successively  Lemma \ref{thm:Lhhv2k}, \eqref{eq:Qgffv2k} of
Corollary \ref{thm:Qgfvk} and Lemma \ref{thm:residual}, it holds for
$R>3$, $k>3+3/2$  that 
\begin{equation}
  \label{eq:Lfhh}
 \br{\cL_{f(t)}^R\, h,h\wei^{2k}}+K\|h\,\psi^R\wei^k\|_{\H1A}^2\,\le\,
 C\,\left(1+\|f\|_{L^2_3}^4+\|f\|_{H^2_{k-1}}\right)\,\|h\|_{L^2_k}^2\,\le\, C_1\|h\|_{L^2_k}^2,
 \end{equation}
 and
 $$
 \begin{cases}
\ds\br{\cQ^R(g,g),g\wei^{2k}}\,=\;\br{\cQ(g\, \psi^R,g\, \psi^R),g\, \psi^R\wei^{2k}}\,\le\, C_0 \,\|g\|_{L^2_k}\,\left(\|g\,\psi^R\wei^k\|_{\H1A}^2\,+\,\|g\|_{L^2_k}^2\right),
   \\[0.9em]
\ds \|\cR(f(t))\|_{L^2_k}\;\le\, \frac{C}{R^{l}}\,\|f(t)\|_{H^2_{k+l}}^2\,\le\, \frac{C_2}{R^{l}},
\end{cases}
$$
where $K$ only depends on $\|f_0\|_{L^1_2}$, $\cH_0$ and the constants
$C_0$ , $C_1$, $C_2$ depends only on $M$, $k$, $l$. Here with the
normalization $\|f_0\|_{L^1}=1$, $\|f_0\|_{L^1_2}=3$ and $\cH_0\ls
\|f_0\|_{L^2} \,\ls\, M$ we could regard $K$ as a constant that
depends only on $M$.

Putting these together, we are led to that
\begin{eqnarray*}
\ds \f{\dD}{\dD t}\|g\|_{L^2_k}^2\,+\,2\,K\,\|g\,\psi^R\wei^k\|_{\H1A}^2
  &\le&\ds 2\,C_0\,\|g\|_{L^2_k}\|g\,\psi^R\wei^k\|_{\H1A}^2
  \\
  &+& \ds\,2\,\rr{C_1\,+\,C_0\,\|g\|_{L^2_k}\,+\,\f{1}{2}C_2^2}\,\|g\|_{L^2_k}^2\,+\,\frac{1}{R^{2l}}.
\nonumber
\end{eqnarray*}
Observing that  for sufficiently large $R>0$,
\beq
\|g|_{t=0}\|_{L^2_k}^2=\|f_0(1-\psi^R)\|_{L^2_k}^2\,\le\,
\left(\frac{2}{R}\right)^{2l} \,\|f_0\|_{L^2_{k+l}}^2\,\le\; \frac{2^{2l}\,M^2}{R^{2l}}\;\ll\,1
\eeq
and by continuity of $t\to \|g(t)\|_{L^2_k}$, there exists $T_*>0$
such that
$$
T_*\,:=\,\sup\left\{t\in[0,T],\qquad
  \sup_{s\in[0,t]}\|g(s)\|_{L^2_k}\,\le\,
  \min\left(\frac{K}{4C_0},1\right) \right\}.
$$
Then setting $\kappa_0:= C_1+C_0+\f{1}{2}C_2^2$, it yields that
\ben
\label{eq:ddtg11}
\f{\dD}{\dD t}\|g\|_{L^2_k}^2\,+\,K\,\|g\,\psi^R\wei^k\|_{\H1A}^2\le 2\,\kappa_0\,\|g\|_{L^2_k}^2\,+\,\frac{1}{R^{2l}},\quad \forall \,t\,\in[0,T_*],
\een
 and from the  Gronwall's inequality, we get that
$$
\|g(t)\|_{L^2_k}^2 \,\le\,
\rr{\|g|_{t=0}\|_{L^2_k}^2+\frac{1}{R^{2l}}\,\int_0^t e^{-2\kappa_0s}\,\dD s}\,e^{2\kappa_0t}\,\le\, \frac{\CC_0^2}{R^{2l}}\,e^{2\kappa_0t} ,\quad\,\forall t\,\in[0,T_*]\,,
$$
where  $\CC_0=\sqrt{2^{2l}M^2+1/(2\kappa_0)}$.  Therefore,  this
latter estimate shows that  choosing  again $R$ large  enough such that
\beq
\frac{\CC_0}{R^{l}}\, \exp\left(\kappa_0 T\right)\;\le\, \min\left(\frac{K}{4\,C_0},1\right),
\eeq 
which is equivalent to take $R\,>\,\RR_0(T)$ with
$$
\RR_0(t)\,:=\,\rr{\f{\CC_0\,\exp\left(\kappa_0\,t\right) }{\min\{K/(4C_0),1\}}}^{1/l}.
$$
thus, \eqref{errorbetweenfrf}  holds with $n=0$ for  $T_*=T$. Moreover,
integrating \eqref{eq:ddtg11} between $0$ and $t$ and using the latter
inequality, we also  get that 
$$
\int_0^t\|g(s)\,\psi^R\wei^k\|_{\H1A}^2\,\dD s\,\ls\, 1+t,\ \ \ \forall
t\in[0,T].
$$ 
Using a classical iteration scheme and the previous estimate, we prove
the well-posedness to \eqref{eq:trg}, hence to  the truncated equation \eqref{eq:trf}. 

Now, we proceed by induction  to prove
  \eqref{errorbetweenfrf} for $n\geq 1$ and suppose that  there exist
  positive constants $\CC_{n-1}$, $\kappa_{n-1}$ and a function
  $\RR_{n-1}(t)\,:=\, C \,e^{(\kappa_{n-1}/l)t}$ such that for $R>\RR_{n-1}(T)$, 
  \ben
  \label{inductivehy1}
  \begin{cases}
   \ds \|g(t)\|_{H^{n-1}_k}\,\le\,
    \frac{\CC_{n-1}}{R^{l}}\,e^{\kappa_{n-1}t}\le 1,\,\, \forall
    t\in[0,T], \\[0.9em]
   \ds K\int_0^t   \|g(s)\|_{\mathsf{d}}^2\,\dD s\,\le\, C \,( 1+t),\,\,
    \forall t\in[0,T],
  \end{cases}
  \een
  where
  \ben
  \|g(t)\|_{\mathsf{d}}^2\,:=\,\sum_{|\alpha|\le n-1}\|\pa^\alpha
  g(t)\,\psi^R\wei^k \|_{\H1A}^2.
  \label{inductivehy2}
  \een  
We have already proved this result for $n=0$ in the well-posedness
part. Now  for a   multi-index $a\in\N^3$ with $|a|=n$, we  have
$$
\f12\,\f{\dD}{\dD t}\|\pa^a g\|_{L^2_k}^2\,=\,  \sum_{i=0}^5 \cJ_i,$$
where the term $\cJ_0$ is related to $\cR(f)$,
 $$
\cJ_0 \,:=\,\br{\pa^a\cR(f),\,\pa^a g\,\wei^{2k}},
$$
while the other terms depend on $g$,  for $i=1,...,5$
$$
  \ds \cJ_i\,:=\;  \sum_{\alpha\in\cA_i } \sum_{(h_1,h_2)\in\cE} 
  \br{\cQ(\pa^{\alpha_{11}}h_1\,\pa^{\alpha_{12}}\psi^R,\,\pa^{\alpha_{21}}h_2\,\pa^{\alpha_{22}}\psi^R),\,\pa^{\alpha_{3}}\psi^R
       \,\pa^a g\,\wei^{2k}} ,
 $$
 with
$$
\cE \,:=\, \left\{ (f,g), \, (g,f), \, (g,g)\right\},
$$
and
$$
\alpha\in \cA\,:=\,\left\{ (\alpha_{11},\alpha_{12},\alpha_{21},\alpha_{22},\alpha_{3})
 \in(\N^3)^5,\quad
 \alpha_{11}+\alpha_{12}+\alpha_{21}+\alpha_{22}+\alpha_{3}\,=\,a\right\},
 $$
 while  the sets $\cA_1$,..., $\cA_5$ are given by
 $$
 \left\{
\begin{array}{l}
\ds\cA_1 \,:=\, \left\{ \alpha\in \cA, \quad \alpha_{21}= a \right\},
  \\[0.9em]
 \ds\cA_2 \,:=\, \left\{ \alpha\in \cA, \quad
  (\alpha_{22},\alpha_3)=(0,0), \, \alpha_{21}\neq a  \right\},
  \\[0.9em]
 \ds\cA_3 \,:=\, \left\{ \alpha\in \cA, \quad
  (\alpha_{12},\alpha_3)=(0,0), \, |\alpha_{22}|\geq 1   \right\},
\end{array}\right.
  $$
  and
  $$
\ds\cA_4 \,:=\, \left\{ \alpha\in \cA, \quad
  (\alpha_{12},\alpha_{22})=(0,0), \, |\alpha_{3}|\geq 1   \right\},
$$
whereas $\cA_5 = \cA\setminus \cup_{i=1}^4\cA_i$.

We first consider the term $\cJ_1$, which is a  special
situation since it provides the coercivity estimate.  Indeed, choosing
successively $(h_1,h_2)=(f,g)$ and  $(h_1,h_2)=(g,f)$ and using \eqref{eq:Lfhh}, we derive that 
\begin{eqnarray*}
  && \ds\br{\cQ(f\psi^R,\,\pa^{a}g \,\psi^R)\, , \psi^R \,\pa^a
  g\,\wei^{2k}} \,+\,  \br{\cQ(\pa^{a}g\, \psi^R,\,f\,\psi^R),\, \psi^R
  \,\pa^a g\,\wei^{2k}} \\
  &&\ds\,=\;\br{\cL_{f(t)}^R\,
                             \pa^{a}g,\,\pa^{a}g\,\wei^{2k}}
  \\
   &&  \ds\,\le\, -K\;\|\pa^{a}g\,\psi^R\,\wei^k\|_{\H1A}^2  \,+\,
        C_1\,\|\pa^{a}g\|_{L^2_k}^2\,,
        \end{eqnarray*}
while for $(h_1,h_2)=(g,g)$ and applying  \eqref{eq:Qgffv2k} of Corollary \ref{thm:Qgfvk}, it holds that
\begin{equation*}
\left|\br{\cQ(g\,\psi^R,\,\pa^{a}g \;\psi^R),\,\psi^R
    \,\pa^a g\,\wei^{2k}}\right|\,\le\, C_0
\,\|g\|_{L^2_k}\,\left(\,\|\pa^a g \,\psi^R\wei^k\|_{\H1A}^2\,+\, \|\pa^a g\|_{L^2_k}^2\,\right).
\end{equation*}
Gathering these estimates, it yields
\begin{eqnarray*}
\cJ_1 & \le& -K\|\pa^{a}g\,\psi^R\wei^k\|_{\H1A}^2  + C_1\|\pa^{a}g\|_{L^2_k}^2+C_0 \|g\|_{L^2_k}(\|\pa^a g \,\psi^R\wei^k\|_{\H1A}^2+\|\pa^a g\|_{L^2_k}^2)\\
&\le& -\left(K\;-\;C_0
    \|g\|_{L^2_k}\right)\,\|\pa^{a}g\,\psi^R\,\wei^k\|_{\H1A}^2\;+\,C_0\,\|g\|_{L^2_k}\,\|\pa^a g\|_{L^2_k}^2.
\end{eqnarray*}
 Then we consider the term $\cJ_2$ using \eqref{eq:Qgfhv2k3} of Lemma
\ref{thm:upperboundQgfh},  we  have
\begin{eqnarray*}
  &&\ds\left|\br{\cQ(\pa^{\alpha_{11}}h_1\,
  \pa^{\alpha_{12}}\psi^R,\,\pa^{\alpha_{21}}h_2 \,\psi^R), \;\psi^R
  \,\pa^a g\,\wei^{2k}}\right| 
  \\[0.8em]
  &&\ds \,\ls\,\|\pa^{\alpha_{11}} h_1\,\pa^{\alpha_{12}} \psi^R\|_{L^2_k}\,\|\pa^{\alpha_{21}}h_2 \,\psi^R \,\wei^k\|_{\H1A}\,\left(\| \pa^a g
  \,\psi^R \,\wei^k\|_{\H1A}
                        \,+\, \| \pa^a g \,\psi^R\|_{L^2_k}\right)\\
 && \ds \,\ls\,\|\pa^{\alpha_{11}} h_1\|_{L^2_k}\,\|\pa^{\alpha_{21}}h_2 \,\psi^R \,\wei^k\|_{\H1A}\,\left(\| \pa^a g
  \,\psi^R \,\wei^k\|_{\H1A}
                        \,+\, \| \pa^a g \,\psi^R\|_{L^2_k}\right).
\end{eqnarray*}
Since $\alpha\in\cA_2$, we know that  $|\alpha_{21}|\le n-1$  and
$|\alpha_{11}|\le n$, hence we have
$\|\pa^{\alpha_{21}}g \,\psi^R \,\wei^k\|_{\H1A}\,\le\,
\|g\|_{\mathsf{d}}\ls \|g\|_{H^n_k}$ as \eqref{eq:h1s} in Lemma
\ref{thm:fs2}.  Thus, using that  $\|f\|_{H^n_k}\ls \|f\|_{H^{n+2}_{k+l}}\,<\,M$ and $\| \pa^a g \,\psi^R\|_{L^2_k}\le \| \pa^a g\|_{L^2_k}\le \|g\|_{H^n_k}$ it holds that
\begin{eqnarray*}
\cJ_2&\ls& \ds\sum_{\alpha_{21}\neq a}\left(\|\pa^{\alpha_{11}}
           f\|_{L^2_k}\,\|\pa^{\alpha_{21}}g \psi^R
           \wei^k\|_{\H1A}+\|\pa^{\alpha_{11}}
           g\|_{L^2_k}\,\|\pa^{\alpha_{21}}f \,\psi^R
           \,\wei^k\|_{\H1A}+\|\pa^{\alpha_{11}}
           g\|_{L^2_k}\,\|\pa^{\alpha_{21}}g \,\psi^R
           \,\wei^k\|_{\H1A}\right)\\
     &&                                    \times\,\left(\| \pa^a g  \,\psi^R \,\wei^k\|_{\H1A}\,+\, \| \pa^a g \,\psi^R\|_{L^2_k}\right)\\[0.8em]
  &\ls& \ds\rr{\|f\|_{H^n_k}\|g\|_{H^n_k}+\|g\|_{H^n_k}\|f\|_{H^n_k}+\|g\|_{H^n_k}\|g\|_{\mathsf{d}}}\left(\| \pa^a g\,\psi^R \,\wei^k\|_{\H1A}\,+\, \| g\|_{H^n_k}\right)\\[0.8em]
  &\ls & \ds\|g\|_{H^n_k}\,\left(M+\|g\|_{\mathsf{d}}\right)\,\left(\| \pa^a g\,\psi^R \,\wei^k\|_{\H1A}\,+\, \| g\|_{H^n_k}\right)\\[0.8em]
  &\ls & \ds\varepsilon\, \| \pa^a g\,\psi^R \,\wei^k\|_{\H1A}\,+\,C_\varepsilon \,(1+\|g\|_{\mathsf{d}}^2)\,\|g\|_{H^n_k}^2,
\end{eqnarray*}
for any $\varepsilon>0$ which will be determined later. 

Next we estimate the term $\cJ_3$ observing that when $\alpha\in\cA_3$, the term
$\pa^{\alpha_{22}}\psi^R$ is localized in the region
$$
\cC_R \; :=\, \left\{v\in\R^3, \quad R/2\leq |v|\leq R\right\}.
$$
Thus, we notice that both  functions $\pa^{\alpha_{21}}h_2$ and $\pa^a g$
are localized in the region $\cC_R$, hence  together with Lemma
\ref{thm:upperboundQgfh} - \eqref{eq:Qgfhv2k4} and Lemma \ref{thm:fs2}
- \eqref{eq:h1s},  we  have 
\begin{eqnarray*}
 && \left|\br{\cQ(\pa^{\alpha_{11}}h_1\,  \psi^R,\,\pa^{\alpha_{21}}h_2
  \;\pa^{\alpha_{22}}\psi^R), \,\psi^R \,\pa^a g\,\wei^{2k}}\right|
  \\
  &&\ds\,\ls\,  \|\pa^{\alpha_{11}}h_1  \,\psi^R\|_{L^2_k}
        \,\|\pa^{\alpha_{21}}h_2 \,\pa^{\alpha_{22}}\psi^R
        \,\wei^k\|_{H^1_{-1/2}} \,\| \pa^a g \,\psi^R
        \;\wei^k\|_{\H1A}\\
   &&\ds \,+\,  \|\pa^{\alpha_{11}}h_1  \;\psi^R\|_{L^2_k}
       \,\|\pa^{\alpha_{21}}h_2  \,\pa^{\alpha_{22}}\psi^R
       \|_{H^1_{k-2}} \,\| \pa^a g \,\psi^R\|_{L^2_k}.
             \end{eqnarray*} 
             Using that, when $|\alpha_{22}|\geq  1$,  we have $|\pa^{\alpha_{22}}\psi^R|\ls R^{-1}$,  it yields that 
             $$\|\pa^{\alpha_{21}}h_2 \,\pa^{\alpha_{22}}\psi^R
             \,\wei^k\|_{H^1_{-1/2}} \,\ls\,
             \frac{1}{R^{1/2}}\,\|\pa^{\alpha_{21}}h_2
             \,\pa^{\alpha_{22}}\psi^R\|_{H^1_k}\,\ls\,
             \frac{1}{R^{3/2}}\,\|\pa^{\alpha_{21}}h_2\|_{H^1_k},
             $$
             and similarly
             $$
             \|\pa^{\alpha_{21}}h_2  \,\pa^{\alpha_{22}}\psi^R
             \|_{H^1_{k-2}} \,\ls\,
             \frac{1}{R^{3}}\,\|\pa^{\alpha_{21}}h_2\|_{H^1_k},
             $$
             we get that
             \begin{eqnarray*}
 && \left|\br{\cQ(\pa^{\alpha_{11}}h_1\,  \psi^R,\,\pa^{\alpha_{21}}h_2
  \,\pa^{\alpha_{22}}\psi^R), \,\psi^R \,\pa^a g\,\wei^{2k}}\right|
               \\
  && \ds\,\ls\,    \frac{1}{R^{3/2}}\,
             \,\|\pa^{\alpha_{11}}h_1\,\psi^R\|_{L^2_k}
             \,\|\pa^{\alpha_{21}}h_2\|_{H^1_k} \,\left(\| \pa^a g
             \,\psi^R \,\wei^k\|_{\H1A} \,+\,\| \pa^a
         g\|_{L^2_k}\right).
\end{eqnarray*}
Using that $\|g\|_{H^{n-1}_k}\le \|g\|_{H^n_k}$ it holds that
\begin{eqnarray*}
\cJ_3&\ls& \f{1}{R^{3/2}} \rr{\|f\|_{H^{n-1}_k}\|g\|_{H^n_k}+\|g\|_{H^{n-1}_k}\|f\|_{H^n_k}+\|g\|_{H^{n-1}_k}\|g\|_{H^n_k}}\left(\| \pa^a g\,\psi^R \,\wei^k\|_{\H1A}\,+\, \| \pa^a g \,\psi^R\|_{L^2_k}\right)\\
&\ls & \|g\|_{H^n_k}(M+\|g\|_{H^{n-1}_k})\left(\| \pa^a g\,\psi^R \,\wei^k\|_{\H1A}\,+\, \| g\|_{H^n_k}\right)\\
  &\ls & \varepsilon\, \| \pa^a g\,\psi^R \,\wei^k\|_{\H1A} \,+\,C_\varepsilon (1+\|g\|_{H^{n-1}_k}^2)\|g\|_{H^n_k}^2,
\end{eqnarray*}
for any $\varepsilon>0$ which will be determined later.

Now we consider $\cJ_4$ observing again that
$\pa^{\alpha_{21}}h_2$  is  localized in
the region $\cC_R$ we have
\begin{eqnarray*}
\|\pa^{\alpha_{21}}h_2  \,\psi^R \|_{H^2_{k-3/2}} &\ls&
\|\na(\pa^{\alpha_{21}}h_2)\,  \psi^R \|_{H^1_{k-3/2}}
                                                        \;+\;\|\pa^{\alpha_{21}}h_2   \|_{H^1_{k-3/2}} \\
  &\ls& \|\na(\pa^{\alpha_{21}}h_2)\,\psi^R
\;\wei^k\|_{\H1A}\,+\,\frac{1}{R^{3/2}}\,\|\pa^{\alpha_{21}}h_2 \|_{H^1_k}.
\end{eqnarray*}

Then using that $\partial^a g\, \pa^{\alpha_{3}}\psi^R$  is also concentrated in
the region $\cC_R$ and  $|\pa^{\alpha_{3}}\psi^R|\ls R^{-1}$,  Lemma
\ref{thm:upperboundQgfh} - \eqref{eq:Qgfhv2k2} yields that 
\begin{eqnarray*}
 && \left|\br{\cQ(\pa^{\alpha_{11}}h_1\,  \psi^R,\,\pa^{\alpha_{21}}h_2
  \,\psi^R),\, \pa^{\alpha_{3}}\psi^R \,\pa^a
  g\,\wei^{2k}}\right| \\
  &&\qquad\ds \ls\, \|\pa^{\alpha_{11}}h_1  \psi^R\|_{L^2_k}
     \,\|\pa^{\alpha_{2,1}}h_2  \,\psi^R \|_{H^2_{k-1}}\,\| \pa^a
     g \,\pa^{\alpha_{3}}\psi^R  \|_{L^2_k} \\
  &&\qquad\ds\ls\, \frac{1}{R^{1/2}} \, \|\pa^{\alpha_{11}}h_1
     \,\psi^R\|_{L^2_k}  \,\|\pa^{\alpha_{21}} h_2 \, \psi^R
     \|_{H^2_{k-3/2}} \,\| \pa^a g   \|_{L^2_k}, 
     \end{eqnarray*}
     that is,
     \begin{eqnarray*}
 && \left|\br{\cQ(\pa^{\alpha_{11}}h_1\,  \psi^R,\,\pa^{\alpha_{21}}h_2
  \,\psi^R),\, \pa^{\alpha_{3}}\psi^R \,\pa^a
       g\,\wei^{2k}}\right| \\
       && \ds\ls\,   \|\pa^{\alpha_{11}}h_1 \, \psi^R\|_{L^2_k}
     \left( \frac{1}{R^{1/2}} \;\|\na(\pa^{\alpha_{21}}h_2)\,  \psi^R
     \,\wei^k\|_{\H1A}\,+\,\frac{1}{R^{2}}\,\|\pa^{\alpha_{21}}h_2
       \|_{H^1_k}\right) \,\| \pa^a g   \|_{L^2_k}.
     \end{eqnarray*}
     Since $\alpha\in\cA_4$, we have that
     $$
     \|\na(\pa^{\alpha_{21}}g)\,  \psi^R \,\wei^k\|_{\H1A}\le
     \|g\|_{\mathsf{d}}+\sum_{|\beta|=n}\|\pa^{\beta}g\,\psi^R\,\wei^k\|_{\H1A},
     $$
     and $$
     \|\na(\pa^{\alpha_{21}}f)\,  \psi^R \,\wei^k\|_{\H1A}\ls
     \|f\|_{H^{n+2}_k}\le M.
     $$
     Hence by similar calculations, we get
\begin{eqnarray*}
\cJ_4&\ls&\ds
           \rr{\|f\|_{H^{n-1}_k}+\|g\|_{H^{n-1}_k}}\rr{\|g\|_{\mathsf{d}}+\sum_{|\beta|=n}\|\pa^{\beta}g\,\psi^R\,\wei^k\|_{\H1A}+\|g\|_{H^n_k}}\|
           \pa^a g \|_{L^2_k}
           \\ && \ds +\,\,\|g\|_{H^{n-1}_k}\,\|f\|_{H^{n+2}_k}\| \pa^a g \|_{L^2_k}\\
&\ls &\ds (M+\|g\|_{H^{n-1}_k})\rr{\sum_{|\beta|=n}\|\pa^{\beta}g\,\psi^R\,\wei^k\|_{\H1A}+\|g\|_{H^n_k}}\|g\|_{H^n_k}\;+\;M\;\|g\|_{H^n_k}^2\\
  &\ls &\ds \varepsilon \,\sum_{|\beta|=n}\|\pa^{\beta}g\,\psi^R\,\wei^k\|_{\H1A}\,+\,C_\varepsilon\, \,\left(1+\|g\|_{H^{n-1}_k}^2\right)\,\|g\|_{H^n_k}^2,
\end{eqnarray*}
for  any $\varepsilon>0$ which will be determined later.

Finally the last term $\cJ_5$, we observe that both of functions
$\pa^{\alpha_{21}}h_2$ and $\pa^a g$ are localized in
$\cC_R$. Moreover, we observe that  $|\alpha_{11}|$  and
$|\alpha_{21}|$ are both less or equal than $n-2$, hence applying  Lemma
\ref{thm:upperboundQgfh} -  \eqref{eq:Qgfhv2k2}, we get that
\begin{eqnarray*}
  && \left|
  \br{\cQ(\pa^{\alpha_{11}}h_1\,\pa^{\alpha_{12}}\psi^R,\,\pa^{\alpha_{21}}h_2\,\pa^{\alpha_{22}}\psi^R),\,\pa^{\alpha_{3}}\psi^R
  \,\pa^a g\,\wei^{2k}}\right| \\
  &&\ds\,\ls\, \|\pa^{\alpha_{11}}h_1
     \,\pa^{\alpha_{12}}\psi^R\|_{L^2_k} \,\|\pa^{\alpha_{21}}h_2 \,
     \pa^{\alpha_{22}}\psi^R \|_{H^2_{k-1}}\, \| \pa^a g
     \,\pa^{\alpha_{3}}\psi^R  \|_{L^2_k},
     \end{eqnarray*} 
     which yields 
     \begin{eqnarray*}
       && \left|
  \br{\cQ(\pa^{\alpha_{11}}h_1\,\pa^{\alpha_{12}}\psi^R,\,\pa^{\alpha_{21}}h_2\,\pa^{\alpha_{22}}\psi^R),\,\pa^{\alpha_{3}}\psi^R
       \,\pa^a g\,\wei^{2k}}\right| \,
       \\
       &&\;\ls\, \frac{1}{R^2} \, \|\pa^{\alpha_{11}}h_1
     \|_{L^2_k}  \,\|\pa^{\alpha_{21}}h_2   \|_{H^2_{k}}\, \|
     \pa^a g   \|_{L^2_k}.
     \end{eqnarray*}
  Thus,   we have
\begin{eqnarray*}
\cJ_5&\ls&\ds \f{1}{R^2}\,\rr{\|f\|_{H^{n-2}_k}\,\|g\|_{H^{n}_k}\,+\,\|g\|_{H^{n-2}_k}\,\|f\|_{H^{n}_k}\,+\,\|g\|_{H^{n-2}_k}\,\|g\|_{H^{n}_k}}\,\| \pa^a g \|_{L^2_k}\\
  &\ls &\ds \left(1+\|g\|_{H^{n-2}_k}^2\right)\;\|g\|_{H^n_k}^2.
\end{eqnarray*} 
Now  gathering the latter estimates, it yields that
\begin{eqnarray*}
\sum_{i=1}^5 \cJ_i  &\le &\ds-\;\left(K-C_0
    \|g\|_{L^2_k}\right)\,\|\pa^a g\,\psi^R\,\wei^k\|_{\H1A}^2\,+\,\varepsilon
  \,\sum_{|\beta|=n}\|\pa^{\beta}g\,\psi^R\,\wei^k\|_{\H1A}^2\\ 
  &&\ds +\,C_\varepsilon\,(1+\|g\|_{\mathsf{d}}^2+\|g\|_{H^{n-1}_k}^2)\,\|g\|_{H^n_k}^2.
\end{eqnarray*}
Using  that $\|g\|_{H^n_k}^2\ls \sum_{|\beta|=n}\| \pa^\beta g \|_{L^2_k}+\|g\|_{H^{n-1}_k}$ and the induction hypothesis $\|g\|_{H^{n-1}_k}\le 1$, it follows that
\begin{equation}
  \label{estimateT123}
  \ba
\sum_{i=1}^5 \cJ_i   &\le&\ds-\left(K-C_0
    \|g\|_{L^2_k}\right)\,\|\pa^{a}g\,\psi^R\,\wei^k\|_{\H1A}^2\,+\,\varepsilon
  \,\sum_{|\beta|=n}\|\pa^{\beta}g\,\psi^R\,\wei^k\|_{\H1A}^2\\ 
 \; &\,&\ds \;+\,C_\varepsilon\,(1+\|g\|_{\mathsf{d}}^2)\,\sum_{|\beta|=n}\| \pa^\beta g \|_{L^2_k}^2\,+\;C_\varepsilon\,(1+\|g\|_{\mathsf{d}}^2)\,\|g\|_{H^{n-1}_k} ,
  \ea
\end{equation}
with some $\varepsilon$ small enough and a constant $C_\varepsilon$
depending only on $\varepsilon$, $M$, $n$, $k$ and $l$.

Now it remains to estimate the term $\cJ_0$ observing that 
\begin{eqnarray*}
\cJ_0&=& \br{\pa^a(\cQ(f(\psi^R-1),f\psi^R)\,\psi^R),\pa^a
      g\wei^{2k}} +\,\br{\pa^a(\cQ(f,f(\psi^R-1))\,\psi^R),\pa^a
        g\wei^{2k}} \\
     &&
\,-\,\br{\pa^a (\cQ(f,f)(1-\psi^R)),\pa^a
        g\wei^{2k}}.
        \end{eqnarray*}
Applying  \eqref{eq:Qgfhv2k2} in Lemma \ref{thm:upperboundQgfh}, we have 
 \begin{eqnarray*}
  && |\cJ_0| \,\ls\, \\
&&                 \left( \|f(\psi^R-1)\|_{H^n_k}\,\|f\psi^R\|_{H^{n+2}_k}\,+\,\|f\|_{H^n_k}\,\|f(1-\psi^R)\|_{H^{n+2}_k}\,+\,\|\pa^a(\cQ(f,f)\,(1-\psi^R))\|_{L^2_k}\right)\|\pa^a
                 g\|_{L^2_k}.
                 \end{eqnarray*} 
 By Leibniz rule, we reduce the terms in the right-hand side to the
 typical terms  $\|\pa^{b}f \;\pa^{c}(\psi^R-1)\|_{L^2_k}$ and $
 \|\pa^{b}\left(\cQ(f,f)\,\pa^{c}(1-\psi^R)\right)\|_{L^2_k}$ for which  it is easy to verify that 
$$
 \|\pa^{b}f \;\pa^{c}(\psi^R-1)\|_{L^2_k}\,\ls\,
 \frac{1}{R^{l}}\, \| \pa^{b}f\|_{L^2_{k+l}},
 $$
 and
 $$
 \|\pa^{b}\left(\cQ(f,f)\,\pa^{c}(1-\psi^R)\right)\|_{L^2_k}\,\ls\,
 \frac{1}{R}\, \|\pa^{b}(\cQ(f,f))\|_{L^2_{k+l}}\,.
 $$
 It yields that
\ben
\label{estimateT4}
|\cJ_0|\,\ls\, \frac{1}{R^{l}} \,
\|f\|_{H^{n+2}_{k+l}}^2\,\|\pa^a g\|_{L^2_k} \,\le\,
\frac{M^2}{R^{l}}\, \|\pa^a g\|_{L^2_k}.
\een  

Thanks to \eqref{estimateT123} and \eqref{estimateT4} and  summing
over  all multi-indexes $|a|=n$, with $K-C_0\|g\|_{L^2}\ge \f34 K$
proved in well-posedness part, and taking $\varepsilon$ small enough and
depending only on $K$ and $n$, we arrive at
\ben
\label{eq:ddtgn}
\ba
 &\f{\dD}{\dD t}\sum_{|a|=n}\|\pa^a
 g\|_{L^2_k}^2+K\sum_{|a|=n}\|\pa^{a}g\,\psi^R\wei^k\|_{\H1A}^2
 \\
 &\le C_K   \;\left( \sum_{|a|=n}\| \pa^a g  \|_{L^2_k}^2 \;+\, \|g\|_{H^{n-1}_k}^2
 \right)\, (1+\|g\|_{\mathsf{d}}^2) \,+\;\frac{1}{R^{2l}}\,,
\ea
 \een
 where $C_K$ depends on $M$, $n$, $l$ and $k$. Then using that $\sup_{t\in[0,T]}\|g\|_{L^2_k}\le \min\{K/(4\,C_0),1\}$
 and by Gronwall's inequality, we get that  for $t\in[0,T]$,
 \beno
 \sum_{|a|=n}\|\pa^a g(t)\|_{L^2_k}^2&\le&
 e^{C_K\;t+C\,(1+t)}\; \left[ \sum_{|a|=n}\|\pa^a
     g(0)\|_{L^2_k}^2\,+\,\frac{1}{R^{2l}}\,\left(2\,C_K
       \,\CC_{n-1}\,e^{\kappa_{n-1}t}\;C \;(1+t)+1\right)\right]
 \\
 &\le& 4\;\left(2\;C_K \,\CC_{n-1}^2\;e^{2\kappa_{n-1}\,t}\,C\,(1+t)\,+\,M^2\right)\;e^{C(1+t)}
    \frac{e^{C_Kt}}{R^{2l}}
 \\ & \le & \frac{\bar{\CC}_n^2}{R^{2l}}\; e^{2\kappa_n \;t}
 ,
 \eeno  
where the constants $\bar{\CC}_n$ and $\kappa_n>\kappa_{n-1}$ depend
only on $M$, $n$, $k$ and $l$. Applying Lemma \ref{WN2lHNl},  we know that
\beno
\|g\|_{H^n_k}^2\;\le\; C \,\rr{\sum_{|a|=n}\|\pa^a
  g(t)\|_{L^2_k}+\|g\|_{H^{n-1}_k}^2} \,\le\, \frac{\CC_n^2}{R^{2l}}
\, e^{2\kappa_n \,t}\,,
\eeno
with some constants $\CC_n$ and $\kappa_n$ depending only on $M$, $n$,
$k$ and $l$. Thus, taking $R>0$ such that
\begin{equation}
  \label{def:Rn}
  R>\RR_n(T):=(\CC_ne^{\kappa_n\,T})^{1/l}, 
\end{equation}
we get that
\ben
\label{eq:ghnk}
\|g\|_{H^n_k}\,\le\, \frac{\CC_n}{R^l} \;e^{\kappa_n\, t} \,\le\; 1.
\een
Substituting \eqref{eq:ghnk} into \eqref{eq:ddtgn} and integrating on
$[0,\,t]$,  we get using the induction hypothesis \eqref{inductivehy1} and \eqref{inductivehy2}
\beq
K\int_0^t \sum_{|a|=n}\|\pa^{a}g(s)\,\psi^R\wei^k\|_{\H1A}^2\, \dD s\,\ls\; 1+t,\qquad \forall \;t\,\in\;[0,T].
\eeq
 Thus, again with induction hypothesis \eqref{inductivehy2} we have
\beq
K\int_0^t \sum_{|a|\le n}\|\pa^{a}g(s)\,\psi^R\wei^k\|_{\H1A}^2\, \dD s\ls 1+t,\qquad \forall \;t\;\in\;[0,T],
\eeq
which we completes the induction for case $n$.

\subsection{Nonnegativity} 
\label{sec33}

 To conclude the proof of Theorem \ref{thm:truncationresult}, we now
 establish the nonnegativity of $f^R$ for sufficiently large $R$. For
 a fixed $T>0$, we  suppose that $R > \RR_3(T)$ with $\RR_3(T)$
 derived from previous proof of the regularity part. By
 \eqref{eq:ghnk}, we have
 $$
 \|f - f^R\|_{H^3_k} \;\leq\; 1, \quad\forall\; t \;\in\; [0,T].
 $$
 Hence, by Sobolev embedding theorem, it yields that   $f^R(t) \;\in\; \mathscr{C}^{1,1/2}(\mathbb{R}^3)$.

Now we define $f^R_- \,:=\; f^R \mathbf{1}_{\{f^R < 0\}}$, noting that $\nabla f^R_- = \nabla f^R \mathbf{1}_{\{f^R < 0\}}$ is well-defined in $L^1_{loc}(\R^3)$ since $\nabla f^R$ is continuous. The energy estimate yields:

\begin{eqnarray*}
\f{1}{2}\f{\dD}{\dD t}\|f^R_-\|_{L^2}^2&=&\br{\f{\dD}{\dD
    t}f^R,f^R\mathbf{1}_{\{f^R<0\}}}
\\
\, &=&\int_{\R^3}\cQ^R(f^R,f^R)\,f^R\mathbf{1}_{\{f^R<0\}} \,\dD v\,=\,\int_{\R^3}\cQ(f^R\psi^R,f^R\psi^R)\,f^R\psi^R\,\mathbf{1}_{\{f^R<0\}}\, \dD v.
\end{eqnarray*}
By further computation,   it holds that
$$
\ba
\int_{\R^3}\cQ(f^R\psi^R,f^R\psi^R)\,f^R\psi^R\,\mathbf{1}_{\{f^R<0\}}\,
\dD v&=& -\int_{\R^3}(a*(f^R\psi^R)):\na(f^R\psi^R)\otimes \na (f^R\psi^R) \, \mathbf{1}_{\{f^R<0\}}\, \dD v\\
\;&\;&+\;4\;\pi\; \int_{\R^3} (f^R\psi^R)^3\,\mathbf{1}_{\{f^R<0\}}\,\dD v.
\ea
$$
We analyze each term of the right hand side separately. On the one
hand, since $\psi^R\ge 0$, one has
$$
\int_{\R^3} (f^R\psi^R)^3\,\mathbf{1}_{\{f^R<0\}}\,\dD v\le 0.
$$
On the other hand, thanks to Proposition 2.1 in \cite{HJL},  for any
$\xi\in\S^2$, there exists a constant $c_0$ only depending on
$\|f\|_{L^1_2}$ and $\cH(f)$  such that for $v\in\R^3$,
\beq
(a*f)(v):\xi\otimes \xi\,\ge\, \frac{c_0}{\wei^{3}} \; |\xi|^2\,=\,\frac{c_0}{\wei^{3}}.
\eeq
Then, we  write
$$
a * (f^R \psi^R) \,=\, a * f \,+\, a * (f^R \psi^R - f).
$$
and estimate the perturbation term.  We have for any $\xi \in \mathbb{S}^2$:
\begin{eqnarray*}
\left(a*(f^R\psi^R-f)\right)(v):\xi\otimes \xi &=& \int_{\R^3}
                                                   a(v-v_*):\xi\otimes
                                                   \xi\,(f^R(v_*)\psi^R(v_*)-f(v_*))\,\dD
                                                   v_*
  \\
&=&\int_{\R^3} |v-v_*|^{-3}\,|(v-v_*)\times
    \xi|^2(f_*^R\psi^R_*-f_*)\,\dD v_*
  \\
                                               &\le&  \int_{\R^3}|v-v_*|^{-1}|f_*^R\psi^R_*-f_*|\,\dD v_*
  \\
                                               &\ls&  \|f^R\psi^R-f\|_{L^2_2}\\
                                               &\ls&
                                                     \|(f^R-f)\,\psi^R\|_{L^2_2}\,+\,\|f\,(1-\psi^R)\|_{L^2_2}
                                                     \,\ls\,
                                                     \frac{1}{R^{l}}
                                                     \left(\CC_0\,e^{\kappa_0\;t}\,+\;2^{l}\,M\right).
                                                     \end{eqnarray*} 
Combining these estimates, we obtain for any $\xi \in \mathbb{R}^3$ and $l > 3$:
\beq
\ba
  1_{\{|v|\le R\}}\;\left(a*(f^R\psi^R)\right)(v):\xi\otimes \xi\,\ge\,
  \frac{1}{R^{3}} \;\left( c_0\;-\;\frac{C}{R^{l-3}} \;\left(\CC_0 \,e^{\kappa_0\,t}\,+\;2^{l}\,M\right)\right).
\ea
\eeq
where $c_0$, $C$, $\mathcal{C}_0$, and $\kappa_0$ depend only on $M$,
$k$, and $l$. Now we define
\ben
\label{def:Rpos}
\RR_{pos}(t) \;:=\;C_{pos}\,  e^{\kappa_{pos}t}\ge\max\left\{\rr{\f{8\,C}{c_0}(\CC_0\,e^{\kappa_0\,t}\,+\;2^{l}\;M)}^{\f{1}{l-3}},\RR_3(t)\right\},
\een
 with $\kappa_{pos}=\max\{\kappa_0/(l-3),\kappa_3/l\}$ and the constant $C_{pos}$ depends only on $M,k,l$. For $R>\RR_{pos}(T)\ge \RR_3(T)$, we have 
\beq
\ba
&\int_{\R^3}(a*(f^R\psi^R))(v):\na(f^R\psi^R)\otimes \na (f^R\psi^R)
\, \mathbf{1}_{\{f^R<0\}}\, \dD v\\
=&\int_{\R^3} 1_{|v|\le R}(a*(f^R\psi^R))(v):\na(f^R\psi^R)\otimes \na (f^R\psi^R) \, \mathbf{1}_{\{f^R<0\}}\, \dD v\;\ge\; 0.
\ea
\eeq
Consequently, choosing  $\tilde\RR(t)$ in \eqref{def:Rtilde} as $ \tilde\RR(t)  =\max\{\RR_n(t),\RR_{pos}(t)\}$,
where $\RR_n(t)$ is derived from the propagation of Sobolev regularity
in \eqref{def:Rn}, and $\RR_{pos}(t)$ arises from the nonnegativity
argument \eqref{def:Rpos} , we  set
$\tilde\kappa=\max\{\kappa_n/l,\kappa_{pos}\}$, which yields that
$$
\f{\dD}{\dD t} \|f^R_-\|_{L^2}^2 \,\leq\, 0.
$$
 Given the nonnegative initial data $f^R|_{t=0} = f_0 \psi$, it
 follows that $\|f^R_-(t)\|_{L^2} = 0$ for all $t \in [0,T]$, proving
 the nonnegativity of $f^R$  and concluding the proof of Theorem \ref{thm:truncationresult}.

\section{Proof of Theorem \ref{thm:fourierresult}: Analysis of the
  spectral method}
\label{sec:4}
\setcounter{equation}{0}
\setcounter{figure}{0}
\setcounter{table}{0}

This section is devoted to the proof of Theorem
\ref{thm:fourierresult} focusing on the upper bounds for the periodic
Landau collision operator $\cQ_\#$ and the error estimates between the
solution of the truncated problem $f^R$ of \eqref{eq:trf} and the numerical
approximation $f^{R,N}_\#$ given by \eqref{eq:fRn}.

\subsection{Upper bounds for the periodic Landau collision operator $\cQ_\#$}
\label{sec41}

We recall that for any periodic functions  $f$, $g\in L^2(\cD_L)$, it holds that
\begin{eqnarray*}
  \br{f,g}_\# &:=& \int_{\cD_L}f(v)g(v)\, \dD v
  \,=\,\frac{1}{(2L)^{3}}\,\sum_{k\in\Z^3}\hat{f}(k)\overline{\hat{g}(k)}\,,
\end{eqnarray*}
hence the $L^2$ norm is given as
$$
 \|f\|^2_{L^2} \,=\, \br{f,f}_\# \,=\,\frac{1}{(2L)^{3}}\,\sum_{k\in\Z^3}|\hat{f}(k)|^2,
$$
Then the standard Sobolev norms can be defined as follows: for $a\ge 0$, 
$$
\|f\|_{H^a_{per}}^2\,:=\, \frac{1}{(2L)^{3}}\,\sum_{k\in\Z^3}|\hat{f}(k)|^2\left( 1+\f{\pi^2|k|^2}{L^2} \right)^{a},
$$
and
\begin{equation}
  \label{eq:dHa}
\|f\|_{\dot{H}^a_{per}}^2\,:=\,\frac{1}{(2L)^{3}}\,\sum_{k\in\Z^3}|\hat{f}(k)|^2\left(\f{\pi|k|}{L} \right)^{2a}.
\end{equation}
 where the projection operator $\cP_N$ given in \eqref{defcPN},
 satisfies for any $a>s$, 
 \ben
 \label{eq:1-Nf}
\|({\rm Id}\,-\;\cP_N)\;f\|_{\dot{H}_{per}^s} \,\ls\, \left( \f{L}{N} \right)^{a-s}\|f\|_{\dot{H}_{per}^a}.
\een

\subsubsection*{Estimates for  periodic Landau collision operator} We
have the following results proved in Appendix \ref{sec:appendix3}.
\begin{prop}\label{prop:betalm}
  Let $l$, $m\in\mathbb{Z}^3$. Then 
  \ben
    \label{eq:betalm}
  \cQ_\#(e^{i\f{\pi}{L}l\cdot v},e^{i\f{\pi}{L}m\cdot v}) \,=\,
  \beta(l,m)\, e^{i\f{\pi}{L}(l+m)\cdot v},
  \een 
where $\beta(l,m)$ is given as 
$$
\begin{cases} 
\ds\f{4\pi}{|l|^4}\left( |l\times m|^2\left(
    \cos(|l|\pi)-\f{\sin(|l|\pi)}{|l|\pi} \right)+2(|l|^4-(l\cdot
  m)^2)\left( 1-\f{\sin(|l|\pi)}{|l|\pi} \right) \right),\,\mbox{if }\,\; |l|\neq 0,
\\[0.9em]
\ds-\f{4\pi^3}{3}|m|^2, \quad\mbox{if}\,\, |l|= 0.
\end{cases}
$$
\end{prop}

This exact expression for $\beta(l,m)$ allows to keep the convolution
structure when we calculate the periodic Landau collision operator as
follows: for any $g$, $h\in L^2(\cD_L)$, we have

\beq
\ba
\cQ_\#(g,h)&=& \ds\frac{1}{(2L)^{6}}\;\sum_{l\in\Z^3}\sum_{m\in\Z^3}\hat{g}(l)\,\hat{h}(m)\,\cQ_\#(e^{i\f{\pi}{L}l\cdot v},e^{i\f{\pi}{L}m\cdot v})
\\[0.85em]
&=& \ds\frac{1}{(2L)^{6}}\,\sum_{l\in\Z^3}\sum_{m\in\Z^3}\hat{g}(l)\,\hat{h}(m)\,\beta(l,m)\,e^{i\f{\pi}{L}(l+m)\cdot v}
\\[0.85em]
&=& \ds\frac{1}{(2L)^{6}}\,\sum_{k\in\Z^3}\sum_{l+m=k}\hat{g}(l)\,\hat{h}(m)\,\beta(l,m)\,e^{i\f{\pi}{L}k\cdot v}.
\ea
\eeq
Thus the Fourier coefficients of $\cQ_\#(g,h)$ are given by
\ben
\label{eq:Qhat}
\hat{Q}_\#(g,h)(k)=(2L)^{-3}\sum_{l+m=k}\hat{g}(l)\hat{h}(m)\beta(l,m).
\een
Also from this proposition we have the following estimates for $\beta(l,m)$:
\ben\label{eq:betalm2}
|\beta(l,m)|\ls 1+\f{|m|^2}{|l|^2},\quad \forall l,m\in\Z^3,\,|l|\neq 0; \ \ \ \mbox{and}\ \ |\beta(0,m)|\ls |m|^2, \quad \forall m\in\Z^3.
\een

Now we are in a position to give the estimates for the periodic Landau operator:
\begin{lem}
  \label{thm:Qpghf}
 If $g\in L^2(\cD_L),h\in H^2_{per}(\cD_L)$, then the following upper bound  holds:
  \ben\label{eq:Qpgh}
  \|\cQ_\#(g,h)\|_{L^2} \,\ls\; \frac{1}{L^{3/2}}
  \,\|g\|_{L^2}\,\|h\|_{L^2}\,+\; L^{1/2}\;\|g\|_{L^2}\;\|h\|_{\dot{H}_{per}^2}.
  \een
 If additionally $h\in \mP_{N}$, we have 
 \ben
 \label{thm:QpgPNh}
 \|\cQ_\#(g,h)\|_{L^2} \,\ls\; \frac{N^2}{L^{3/2}}\; \|g\|_{L^2}\, \|h\|_{L^2}.
  \een
\end{lem} 
\begin{proof}
  For functions $g$, $h\in L^2(\cD_L)$, from \eqref{eq:Qhat} and Parseval's identity, we have
\begin{eqnarray*}
\|\cQ_\#(g,h)\|_{L^2}^2
  &=&\frac{1}{(2L)^{3}}\;\sum_{k\in\Z^3}|\hat{Q}_\#(g,h)(k)|^2
  \\
  &=& \frac{1}{(2L)^{9}}\;\sum_{k\in\Z^3}\left|\,\sum_{l+m=k}\hat{g}(l)\,\hat{h}(m)\,\beta(l,m)\right|^2.
\end{eqnarray*}
From the estimate \eqref{eq:betalm2}, for $k\in\Z^3$ we have
\begin{eqnarray*}
\ds\left|\sum_{l+m=k}\hat{g}(l)\hat{h}(m)\beta(l,m)\right|&\ls& \ds\sum_{\substack{l+m=k
        \\ l\neq 0}} |\hat{g}(l)||\hat{h}(m)|\left(
  1+\f{|m|^2}{|l|^2} \right)+|\hat{g}(0)||\hat{h}(k)| |k|^2
  \\
  &\ls& \ds
        \sum_{l+m=k}|\hat{g}(l)| \hat{h}(m)|+\sum_{\substack{l+m=k
        \\ l\neq 0}} |\hat{g}(l)| |\hat{h}(m)|
  \f{|m|^2}{|l|^2}+|\hat{g}(0)| |\hat{h}(k)| |k|^2.
\end{eqnarray*}
Thus by Cauchy's inequality, we get that
\begin{eqnarray*}
  &&\ds\left|\,\sum_{l+m=k}\hat{g}(l)\,\hat{h}(m)\,\beta(l,m)\right|^2\;\ls
  \\
  &&\;\ds \left|\sum_{l+m=k}|\hat{g}(l)|\,|\hat{h}(m)|\,\right|^2\;+\;\left|\sum_{\substack{l+m=k
        \\ l\neq 0}}|\hat{g}(l)|\,|\hat{h}(m)|\,\f{|m|^2}{|l|^2}\,\right|^2+|\hat{g}(0)|^2\,|\hat{h}(k)|^2\,|k|^4.
\end{eqnarray*}
Thus, $\|\cQ_\#(g,h)\|_{L^2}^2$ could be bounded by three terms as follows:
$$
  \|\cQ_\#(g,h)\|_{L^2}^2 \;\ls\; \frac{1}{L^{9}}\left(\cI_1+\cI_2+\cI_3\right),
  $$
  with
  $$
\left\{\begin{array}{l}
 \ds \cI_1 \;=\,
         \sum_{k\in\Z^3}\left|\sum_{l+m=k}|\hat{g}(l)|\,|\hat{h}(m)|\,\right|^2,
         \\[0.9em]
  \ds\cI_2 \;:= \;\sum_{k\in\Z^3}\left|\sum_{\substack{l+m=k
        \\ l\neq
         0}}|\hat{g}(l)|\,|\hat{h}(m)|\,\f{|m|^2}{|l|^2}\,\right|^2,
         \\[0.9em]
  \ds\cI_3 \;:=\; \sum_{k\in\Z^3}|\hat{g}(0)|^2\,|\hat{h}(k)|^2\,|k|^4,
\end{array}\right.
$$
For $\cI_1$, we apply Young's inequality for convolution on $\Z^3$, to
get $\|\hat{g}*\hat{h}\|_{l^2(\Z^3)}\le \|\hat{g}\|_{l^2(\Z^3)}\,\|\hat{h}\|_{l^1(\Z^3)}$, that is,
\beq
\cI_1\;\le\; \sum_{l\in\Z^3}|\hat{g}(l)|^2\,\left(\sum_{m\in\Z^3}|\hat{h}(m)|\right)^2=(2L)^3\,\|g\|_{L^2}\,\left(\sum_{m\in\Z^3}|\hat{h}(m)|\right)^2,
\eeq
hence applying the  Cauchy's inequality with \eqref{eq:dHa} and  using
\ben
\label{eq:hL1}
|\hat{h}(0)|=\|h\|_{L^1(\cD_L)}\ls L^{3/2}\|h\|_{L^2(\cD_L)},
\een
it yields that
\begin{eqnarray*}
 \left(  \sum_{m\in\Z^3}|\hat{h}(m)|\right)^2&\ls&
                                                   |\hat{h}(0)|^2+\left(
                                                   \sum_{m\neq0}|\hat{h}(m)|\right)^2
  \\
  &\ls& L^{3}\|h\|_{L^2}^2\;+\; \sum_{m\neq
        0}\left(|\hat{h}(m)|\f{\pi^2 |m|^2}{L^2}\right)^2 \sum_{m\neq
        0}\f{L^4}{\pi^4|m|^4}
  \\
  &\ls& L^{3}\|h\|_{L^2}^2+ L^7\|h\|_{\dot{H}^2}^2.
\end{eqnarray*}
Thus we have
\beq
\cI_1\;\ls\;\rr{L^3\|g\|_{L^2}^2}\,\rr{L^{3}\|h\|_{L^2}^2+
  L^7\|h\|_{\dot{H}^2}^2}\,=\, L^9\;\|g\|_{L^2}^2\rr{L^{-3}\|h\|_{L^2}^2+ L\,\|h\|_{\dot{H}^2}^2}.
\eeq

For $\cI_2$, similarly it follows from Young's inequality for convolution on $\Z^3$ that
\beq
\cI_2 \,\le\; \left( \sum_{l\neq 0}\f{|\hat{g}(l)|}{|l|^2} \right)^2\left(  \sum_{m\in\Z^3}|\hat{h}(m)|^2|m|^4\right),
\eeq
hence using the Cauchy's inequality and Parseval identity, we have
\beq
\left( \sum_{l\neq 0}\f{|\hat{g}(l)|}{|l|^2} \right)^2\le \left( \sum_{l\neq 0}|\hat{g}(l)|^2 \right)\left( \sum_{l\neq 0}|l|^{-4} \right)\ls L^3\|g\|_{L^2}^2,
\eeq
also, from \eqref{eq:dHa} it holds
\ben\label{eq:mh4}
\sum_{m\in\Z^3}|\hat{h}(m)|^2|m|^4\ls L^4 \sum_{m\in \Z^3}|\hat{h}(m)|^2\left( \f{\pi |m|}{L} \right)^4\ls L^7\|h\|_{\dot{H}^2}^2.
\een
Thus for $\cI_2$ we conclude that
\beq
\cI_2\;\ls\; \rr{L^3\|g\|_{L^2}^2}\; \rr{L^7\|h\|_{\dot{H}^2}^2}\;=\; L^{10}\,\|g\|_{L^2}^2\,\|h\|_{\dot{H}^2}^2.
\eeq

For $\cI_3$, thanks to \eqref{eq:hL1} and \eqref{eq:mh4}, we have
\beq
\cI_3\,\ls\; \rr{L^3\|g\|_{L^2}^2}\rr{L^7\|h\|_{\dot{H}^2}^2}\;=\;L^{10}\,\|g\|_{L^2}^2\,\|h\|_{\dot{H}^2}^2.
\eeq

Gathering the estimates on $\cI_1$, $\cI_2$ and $\cI_3$, we conclude 
the proof of \eqref{eq:Qpgh}. If additionally $h\in \mP_{N}$, with
$\|h\|_{\dot{H}^2}\ls \f{N^2}{L^2}\|h\|_{L^2}$ we obtain
\eqref{thm:QpgPNh}.
\end{proof}

We also note that, as operator $\cP_N$ is self-adjoint, for any
function $g$, $h$ and $f$ defined on the torus, we have
\ben
\label{eq:QRNghf}
\br{\cQ^R_N(g,h),f}\,=\,\br{\cQ_\#(\,\cP_N(g\psi^R),\cP_N(h\psi^R)\,),\cP_N(\cP_Nf\,\psi^R)}.
\een

\subsection{Error estimates}
\label{sec43}

In this section, we denote by 
$f^R_\#$ and $f^{R,N}_\#$   the solution to the equations \eqref{eq:trfp}
and \eqref{eq:fRn}, respectively, and assume they satisfy all the
conditions and results stated in Theorem
\ref{thm:truncationresult}. Similar to the proof of Theorem
\ref{thm:truncationresult}, we define the error function
$$
g_\#\;:=\, f^{R,N}_\#-f^R_\#,
$$
which satisfies:
\ben\label{eq:gn}
\begin{cases}
  \ds\pa_t g_\#\,=\,\cQ^R_N(f^{R,N}_\#,f^{R,N}_\#)\,-\,\cQ_\#(f^R_\#\psi^R, f^R_\#\psi^R)\,\psi^R, \\[0.9em]
  \ds g_\#(t=0)\,=\, f^{R,N}_\#(0)-f^R_\#(0) \,=\,(\cP_N-{\rm Id})(f^R_\#(0))
  \,=\,(\cP_N-{\rm Id})(f_0\psi^R)\,.
\end{cases}
\een
 The energy method yields:
\ben\label{eq:ddtg3}
\f{1}{2}\f{\dD}{\dD t}\|g_\#(t)\|_{L^2}^2\,=\, \cI_1 \,+\,\cI_2,
\een
where we decompose the right-hand side as
$$
\left\{
\begin{array}{l}
\ds\cI_1 \,:=\, \br{\cQ^R_N(f^{R,N}_\#,f^{R,N}_\#)-\cQ^R_N(f^R_\#,f^R_\#), g_\#},
  \\[0.9em]
 \ds\cI_2 \,:=\, \br{\cQ^R_N(f^R_\#,f^R_\#)-\cQ^R_\#(f^R_\#,f^R_\#), g_\#}.
\end{array}
\right.
$$
We split the first term as $\cI_1=\cI_{11} +\cI_{12} +\cI_{13} $ where
$$
\left\{
\begin{array}{l}
\ds\cI_{11} \,:=\, \langle \cQ_\#(F_N, G_N), G'_N \rangle,
  \\[0.9em]
 \ds\cI_{12} \,:=\, \langle \cQ_\#(G_N, F_N), G'_N \rangle,
  \\[0.9em]
 \ds\cI_{13} \,:=\, \langle \cQ_\#(G_N, G_N), G'_N \rangle,
\end{array}
\right.
$$
where  $F_N:=\cP_N(f^R_\#\,\psi^R)$, $G_N:=\cP_N(g_\#\,\psi^R)$ and
$G'_N:=\cP_N(\cP_Ng_\#\,\psi^R)$ with $F_N$, $G_N$,
$G'_N\in\mP_N$.

Now we are in a position to give the upper bounds for $\cI_{1i}$ for
$i=1$, $2$ and $3$.

For $\cI_{11}$, we decompose it as
$\cI_{11}\;=\;\cI_{111}+\cI_{112}+\cI_{113}$, with
$$
\left\{
  \begin{array}{l}
\ds\cI_{111}=\br{\cQ_\#(f^R_\#\psi^R,G_N),G_N}, \quad
    \\[0.9em]
 \ds\cI_{112}\,=\,\br{\cQ_\#(\,(\cP_N-{\rm Id})(f^R_\#\psi^R),G_N\,),G_N},
\\[0.9em]
\ds\cI_{113}\,=\,\br{\cQ_\#(F_N,G_N),G_N'-G_N}.
\end{array}\right.
 $$
For $\cI_{111}$, thanks to the definition of the periodic operator $\cQ_\#$, we have
\beq
\cI_{111}\,=\, -\int_{\cD_L}(a^L*(f^R_\#\,\psi^R)): \na G_N\otimes \na G_N\, \dD v+\int_{\cD_L}(a^L*\na (f^R_\#\,\psi^R))\,G_N\,\na G_N\,\dD v.
\eeq
On the one hand,  applying Theorem \ref{thm:truncationresult}, we know that
$f^R\geq 0$, since both solution $f^R$ and $f^R_\#$ coincide, it follows that  $f^R_\#$ is also
nonnegative, hence   from the first equality in
\eqref{eq:avv}, we show  that the first term of $\cI_{111}$  is nonpositive. On
the other hand, we study the second term defined in the expression of $\cI_{111}$.  Using an integration by parts and Young inequality, we have
\begin{eqnarray*}
  \left|\int_{\cD_L}(a^L*\na (f^R_\#\psi^R))\,G_N\,\na G_N\,\dD
  v\right| &=& \left|\f{1}{2}\int_{\cD_L}(a^L*\na^2
               (f^R_\#\psi^R))\, G_N^2\,\dD v\right|
  \\
           &\ls&  \|a^L*\na^2 (f^R_\#\psi^R)\|_{L^\infty}\;\|G_N^2\|_{L^1}
  \\
  &\ls&  \|a^L\|_{L^2}\|\na^2 (f^R_\#\psi^R)\|_{L^2}\;\|G_N\|_{L^2}^2 \;\ls\;  R^{1/2}\|f^R_\#\|_{H^2}\|g_\#\|_{L^2}^2,
\end{eqnarray*}
which yields that
$$
\cI_{111}\;\ls\; R^{1/2}\|f^R_\#\|_{H^2}\|g_\#\|_{L^2}^2.
$$ 
To estimate both terms $\cI_{112}$ and   $\cI_{112}$ we first observe from \eqref{eq:1-Nf} and the identity $\left({\rm Id} -
\cP_N\right)\,g_\# \,=\, \left(\cP_N - {\rm Id}\right)\,f^R_\#$ that
\begin{eqnarray*}
\|G_N-G_N'\|_{L^2} \;=\;\|\cP_N(({\rm Id}-\cP_N)\,g_\#\, \psi^R)\|_{L^2} &\le&
\|({\rm Id}-\cP_N)g_\#\|_{L^2}
\\
&=&\|({\rm Id}-\cP_N)f^R_\#\|_{L^2}\,\ls\; \left(\frac{R}{N}\right)^{n}\;\|f^R_\#\|_{\dot{H}^n}.
\end{eqnarray*}
Furthermore, by  Theorem \ref{thm:truncationresult} for $n\ge5$ and  $R>\tilde\RR(T)$, we have
\begin{eqnarray*}
\|f^R_\#\|_{H^{n}_{k+l}(\cD_L)}\,=\,\|f^R\|_{H^{n}_{k+l}(\R^3)}
  &\le& \|f\|_{H^{n}_{k+l}(\R^3)}+\|f-f^R\|_{H^{n}_{k+l}(\R^3)}
  \\
  &\le& \, M+1\;\ls\; M.
\end{eqnarray*}
Therefore, it yields that
\begin{eqnarray*}
  |\cI_{112}| &\ls&
                    \frac{N^2}{R^{3/2}}\,\|(\cP_N-{\rm Id})(f^R_\#\psi^R)\|_{L^2}\;
                    \|G_N\|_{L^2}^2
                    \\
                    &\ls&
                          \frac{R^{n-3/2}}{N^{n-2}}\,\|f^R_\#\|_{H^n}\,
                          \|g_\#\|_{L^2}^2,
\end{eqnarray*}
and
\begin{eqnarray*} 
 |\cI_{113}|&\ls& \frac{N^2}{R^{3/2}} \,\|F_N\|_{L^2}\,\|G_N\|_{L^2}\,
                \|G_N-G_N'\|_{L^2}
  \\
  &\ls& \frac{R^{n-3/2}}{N^{n-2}}\,\|f^R_\#\|_{L^2}\;\|g_\#\|_{L^2}\;\|f^R_\#\|_{H^n}
  \\
            &\ls& \frac{R^{2n-7/2}}{N^{2n-4}}\;+\, R^{1/2}\;\|f^R_\#\|_{L^2}^2\,\|f^R_\#\|_{H^n}^2\;\|g_\#\|_{L^2}^2.
\end{eqnarray*}
Using the fact that $\|f^R_\#\|_{L^2}\ls M$, we are led to 
\begin{eqnarray*}
\cI_{11}&\ls&   \left( R^{1/2}\,+\,\frac{R^{n-3/2}}{N^{n-2}}\right) \,\left(\|f^R_\#\|_{H^n}\,+\,\|f^R_\#\|_{H^n}^4\right)\;\|g_\#\|_{L^2}^2\,+\; \frac{R^{2n-7/2}}{N^{2n-4}}
  \\
        &\ls&
              \left(R^{1/2}+\frac{R^{n-3/2}}{N^{n-2}}\right)\;M^4\,\|g_\#\|_{L^2}\;+\;
              \frac{R^{2n-7/2}}{N^{2n-4}}.
\end{eqnarray*}

We now estimate  $\cI_{12}$ and $\cI_{13}$ again by
\eqref{thm:QpgPNh}, similarly to $\cI_{112}$ and $\cI_{113}$, it gives that
\begin{eqnarray*}
|\cI_{12}|\;=\;|\br{\cQ_\#(G_N,F_N),G'_N}| &\ls&
                                                 R^{1/2}\;\|G_N\|_{L^2}\;\|F_N\|_{H^2}\;\|G'_N\|_{L^2}
  \\
  &\le& R^{1/2}\;\|g_\#\|_{L^2}^2\;\|f^R_\#\psi^R\|_{H^2},
\end{eqnarray*}
and
\begin{eqnarray*}
|\cI_{13}|\;=\; |\br{\cQ_\#(G_N,G_N),G'_N}| &\ls&
                                                  \frac{N^{2}}{R^{3/2}}\,\|G_N\|_{L^2}^2\,\|G_N'\|_{L^2}^2
  \\
                                            &\le&  \frac{N^{2}}{R^{3/2}}\;\|g_\#\|_{L^2}^3.
\end{eqnarray*}
Gathering the latter estimates, we derive that

$$
\cI_1 \;\ls_{M,n,k,l}\; \left( R^{1/2}\;+\;\frac{R^{n-3/2}}{N^{n-2}}\right)\;\|g_\#\|_{L^2}^2\;+\; \frac{R^{2n-7/2}}{N^{2n-4}}\;+\;\frac{N^{2}}{R^{3/2}}\;\|g_\#\|_{L^2}^3.
$$

To evaluate $\cI_2$, we  We apply the following decomposition
\begin{eqnarray*}
\cI_2& =& \ds\cQ^R_N(f^R_\#,f^R_\#)-\cQ_\#(f^R_\#\psi^R,
          f^R_\#\psi^R)\;\psi^R
  \\[0.85em]
  &=&\ds\cP_N(\cP_N(\cQ_\#(F_N,F_N))\,\psi^R)-\cQ_\#(f^R_\#\psi^R,
      f^R_\#\psi^R)\;\psi^R
  \\[0.85em]
&=& \ds\left(\cP_N-{\rm Id}\right)(\cP_N(\cQ_\#(F_N,F_N))\,\psi^R)
    \;+\;\left(\cP_N-{\rm Id}\right)(\cQ_\#(F_N,F_N))\, \psi^R
  \\[0.85em]
     &+&\ds\cQ_\#((\cP_N-{\rm Id})(f^R_\#\psi^R),F_N)\;\psi^R\;+\;\cQ_\#\left(f^R_\#\psi^R,\left(\cP_N-{\rm Id}\right)(f^R_\#\psi^R)\right)\;\psi^R.
\end{eqnarray*}

Using \eqref{eq:1-Nf} and \eqref{eq:Qpgh} of Lemma \ref{thm:Qpghf}, we get that
\begin{eqnarray*}
\|\cI_2\|_{L^2} &\ls& \left(\frac{R}{N}\right)^{n-2} \;\|\cQ_\#(F_N,F_N)\|_{H^{n-2}}
  \;+\;
                      \left(\frac{R}{N}\right)^{n-2}\;\|\cQ_\#(F_N,F_N)\|_{H^{n-2}}
  \\[0.85em]
  &+& R^{1/2}\;\|\left(\cP_N-{\rm Id}\right)(f^R_\#\psi^R)\|_{L^2}\,\|F_N\|_{H^2}
                      \,+\;                          \frac{1}{R^{3/2}}\;\|f^R_\#\psi^R\|_{L^2}\;\|(\cP_N-{\rm Id})(f^R_\#\psi^R)\|_{L^2}
  \\[0.85em]
  &+& R^{1/2}\;\|f^R_\#\psi^R\|_{L^2}\;\|(\cP_N-{\rm Id})(f^R_\#\psi^R)\|_{\dot{H}^2}
  \\[0.85em]
&\ls&
      \frac{R^{n-3/2}}{N^{n-2}}\;\|f^R_\#\|_{H^n}^2\;+\;\frac{R^{n-3/2}}{N^{n-2}}\;\|f^R_\#\|_{H^{n-2}}\;\|F_N\|_{H^2}\;+\;\frac{R^{n-3/2}}{N^{n-2}}\;\|f^R_\#\|_{L^2}\;\|f^R_\#\|_{H^n}
  \\[0.85em]
                &\ls& \frac{R^{n-3/2}}{N^{n-2}}\;\|f^R_\#\|_{H^n}^2\;\ls\; \;\frac{R^{n-3/2}}{N^{n-2}}\;M^2,
\end{eqnarray*}
where we use the fact
$$
\|\cQ_\#(F_N,F_N)\|_{H^{n-2}} \;\ls\; R^{1/2}\;\|F_N\|_{H^n}^2\;\ls\;
R^{1/2}\;\|f^R_\#\|_{H^n}^2,
$$
which follows from the Leibniz rule and \eqref{eq:Qpgh} of Lemma
\ref{thm:Qpghf}. This implies that 
$$
|\cI_2| \;\ls\;    \frac{R^{n-3/2}}{N^{n-2}}\;\|g_\#\|_{L^2} \;\ls\,
\frac{R^{2n-7/2}}{N^{2n-4}} \;+\, R^{1/2}\;\|g_\#\|_{L^2}^2.
$$
Therefore, we are led to 
$$
\f{\dD}{\dD t}\|g_\#\|_{L^2}^2\ls_{M,n,k,l} \left( R^{1/2}\;+\;\frac{R^{n-3/2}}{N^{n-2}}\right)\;\|g_\#\|_{L^2}^2\;+\; \frac{R^{2n-7/2}}{N^{2n-4}}\;+\; \frac{N^{2}}{R^{3/2}}\;\|g_\#\|_{L^2}^3.
$$
 Choosing $N$  sufficiently large such that $N>R$,  it  yields
$$
\f{\dD}{\dD t}\|g_\#\|_{L^2}^2\;\le\; \kappa\left( R^{1/2}\|g_\#\|_{L^2}^2\;+\;\frac{R^{2n-7/2}}{N^{2n-4}}\;+\;\frac{N^{2}}{R^{3/2}}\;\|g_\#\|_{L^2}^3\right),
$$
where $\kappa>0$ is a constant  depending only on $M$, $n$, $k$ and $l$. In this situation, the initial data satisfies
\begin{eqnarray*}
\|g_\#(0)\|_{L^2} &=& \|\left({\rm
                        Id}-\cP_N\right)\;f^R_\#(0)\|_{L^2}
  \\
  &\ls&
        \left(\frac{R}{N}\right)^{n-2}\;\|f^R_\#(0)\|_{H^{n-2}}\\
  &\ls&\left(\frac{R}{N}\right)^{n-2}\;M.
\end{eqnarray*}
Thus we may choose $N$ large enough such that  $$
\|g_\#(0)\|_{L^2}^2\le\; C_0\;\left(\frac{R}{N}\right)^{2n-4}
\;\ll\; 1,
$$
 where $n\ge5$ and $C_0>1$ is a constant  depending only on $M$, $n$,
 $k$ and $l$.  By continuity, let
 $$
 T_*:=\sup\left\{ t\in[0; T], \quad
   \sup_{s\in[0,t]}\frac{N^2}{R^{3/2}}\;\|g_\#(s)\|_{L^2}\;\le\; R^{1/2} \right\},
 $$
 hence on the time interval $[0;T_*]$; we  get that
$$
\f{\dD}{\dD t}\|g_\#\|_{L^2}^2\;\le\; 2\;\kappa \,R^{1/2}\;\|g_\#\|_{L^2}^2\;+\;\kappa\; \frac{R^{2n-7/2}}{N^{2n-4}}.
$$
From this together with Gronwall’s inequality, we derive that
\begin{eqnarray*}
\|g_\#(t)\|_{L^2}^2&\le& \left( \|g_\#(0)\|_{L^2}^2\;+\;\kappa
                      \;\frac{R^{2n-7/2}}{N^{2n-4}}\;\int_0^te^{-2\;\kappa
                      \;R^{1/2}\;s}\,\dD s\right) \; e^{2\kappa R^{1/2}\;t}
  \\[0.85em]
                &\le& \CC \,\left(\frac R N\right)^{2n-4}\;e^{2\;\kappa \;R^{1/2}\;t},\quad \forall \;t\;\in\;[0,T_*],
\end{eqnarray*}
with   $\CC:=1/2+C_0$.
Finally, to ensure that $T^*=T$, we need to choose $N$ big enough such that
\beq
\CC\,\left(\frac R N\right)^{2n-4}\;e^{2\;\kappa \;R^{1/2}\;t}\;\le\; \left(\frac{R}{N}\right)^4\,\le\; 1,
\eeq
which is satisfied as soon as  $N\,>\,R\,\NN(R)$ with
$\NN(R)\;=\;\rr{\CC\; e^{\kappa \;R^{1/2}\;T}}^{\f{1}{n-4}}$.
This implies the desired result \eqref{error2} and then completes the proof.

\section{Numerical simulations}
\label{sec:numerical}
\setcounter{equation}{0}
\setcounter{figure}{0}
\setcounter{table}{0}

In this section, we perform numerical simulations using the Fourier
method \eqref{eq:trfp}. We apply a fourth-order Runge-Kutta scheme for time discretization, and the function $\psi^R$ is given by
$$
\psi^R(v) \,=\,
\left\{
  \begin{array}{l}
 1, \,  \ds {\rm if}\,\;  |v| \,<\, 0.9\, R\\[0.8em]
        \ds 10\;\left(1-\frac{|v|}{R}\right), \,  \ds{\rm if}\,\; 0.9 \;R \,\leq\;
|v|\; \leq\; R \\[0.8em]
  0, \,\ds  {\rm if}\,\;  |v| > R.
  \end{array}\right.
$$
First, we address the case of Maxwellian molecules where $\gamma=0$,
for which we compare the approximate solution with an explicit
solution. Second, we address the Coulomb case, which corresponds to
our study, and observe the long-time behavior of the approximate
solution through relative entropy, high-order moments, and Fisher
information.

%
%
%
\subsection{Maxwellian interactions ($\gamma=0$)}

We first consider the case of Maxwellian molecules $\gamma=0$ in
\eqref{1}-\eqref{eq:aij}. Even if this case is not directly covered by our
study, it has some mathematical interest since explicit solutionsnare
available. Indeed, we choose $f_0$ as
$$
f_0(v) \,=\;  \frac{1}{(2\pi\; S_0)^{3/2}}\; \frac{1-S_0}{2\,S_0^2}\,|v|^2 \;\exp\left(-\frac{|v|^2}{2\,S_0^2}\right),
$$
with $S_0=0.6$ where the exact solution is given by
$$
f(t,v) \,=\; \frac{1}{(2\pi \;S)^{3/2}} \, \left(\frac{5\,S-3}{2S} + \frac{1-S}{2\,S^2}\,|v|^2 \right)\,\exp\left(-\frac{|v|^2}{2\,S}\right)\,,
$$
with $S(t)\,=\,1-0.4\,\exp(-4\; t)$.

We compute the  norms of the error  as
$$
\left\{
  \begin{array}{l}
    \ds e_1(t) = \sum_i |e_i(t)|,
    \\
    \ds e_2(t) = \left(\sum_i |e_i|^2\right)^{1/2},
  \end{array}
\right.
$$
where $e_i(t)=|f(t,v_i) - f^{R,N}_\#(t,v_i)|$ and $f^{R,N}_\#(t)$ is
the approximated solution.

We conducted a series of numerical calculations by varying the two
parameters, the length of the domain $2L$ and the number of modes per
direction $N$. First, we present the results when the
number of modes is large, that is, $N \gg 1$, to illustrate the evolution of the error (on
a logarithmic scale) as a function of $L$ in Figure \ref{fig:1}. As
it is proved in Theorem \ref{thm:mainresult}, for a regular solution,
the order of accuracy increases as  $L$ increases. Specifically, when
$N=48$, the error decreases very rapidly when  $L$ is increasing, up to $L \leq 11$. Beyond this value, the error begins to grow, likely due to the insufficient  number of modes  $N$. Conversely, when $N=56$, the
error in $L$ starts to saturate only when $L >12$. 

Finally, it is worth noting that for this initial data and Maxwellian molecules
$\gamma=0$, the error increases over time as $t$ becomes
large. However,
this growth is only polynomial with respect to $t$, suggesting that our
current estimation for Coulombian potential may not be optimal. 

\begin{figure}[ht!]
\begin{center}
\begin{tabular}{cc}
  \includegraphics[width=8.cm]{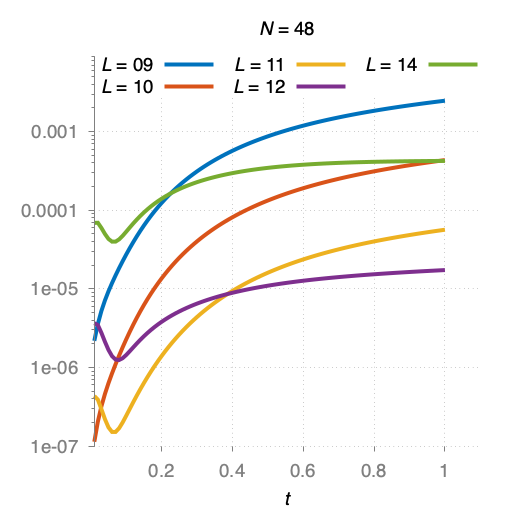} &
  \includegraphics[width=8.cm]{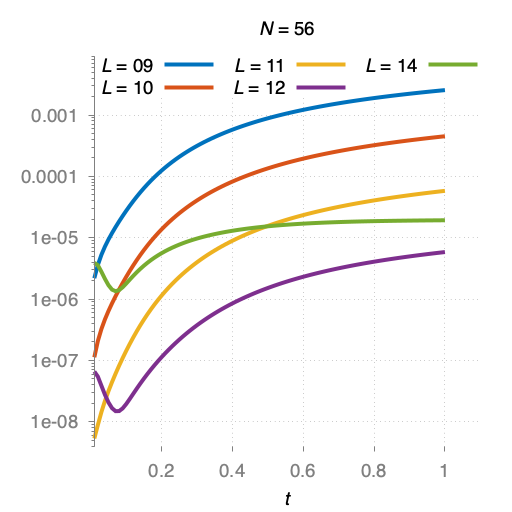}  
\end{tabular}
  \end{center}
\caption{{\bf Maxwellian interactions:} time evolution of  the $L^2$
  error norm in log scale for
  the scheme \eqref{eq:fRn} with respect to $L$ for $N=48$ (left) and $N=56$ (right).}
\label{fig:1}
\end{figure}

In the second part, we consider a sufficiently large domain $\cD_L$
and focus on the error with respect to the number of Fourier modes per
direction $N$. We then plot the evolution
of the error (still on a logarithmic scale) for $L=12$ and
$L=14$ in Figure \ref{fig:2}. First, we observe an increase of  the order of accuracy for
$L=12$ when $N<48$, followed by a saturation of the error for $N=56$. When $L=14$,
this phenomenon is no longer visible for $N\leq 64$. These results
suggest that our error estimate on spectral accuracy, which depends on the parameter $L/N$, is indeed relevant.

\begin{figure}[ht!]
\begin{center}
  \begin{tabular}{cc}
   \includegraphics[width=8cm]{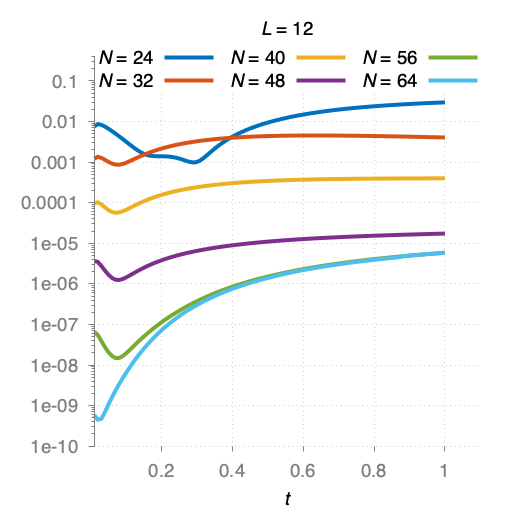} &
   \includegraphics[width=8cm]{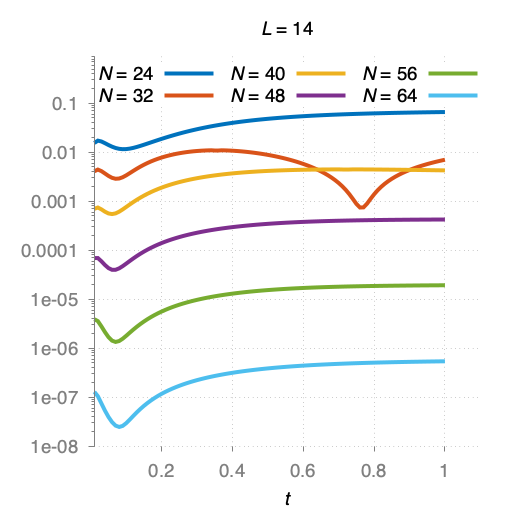}
 \end{tabular}
  \end{center}
\caption{{\bf Maxwellian interactions:} time evolution of  the $L^2$
  error norm in log scale  for
  the scheme \eqref{eq:fRn} with respect to $N$ for $L=12$ (left) and $L=14$ (right).}
\label{fig:2}
\end{figure}

To illustrate this, we plot the maximum error over the interval
$[0,1]$ as a function of the parameter $L/N$ in Figure
\ref{fig:3}. We observe that the order of accuracy increases
exponentially fast, which indicates spectral accuracy. Finally, it is
important to note that the error does not appear to grow as quickly as
suggested by our estimate (Theorem \ref{thm:mainresult}).

\begin{figure}[ht!]
\begin{center}
\includegraphics[width=10cm]{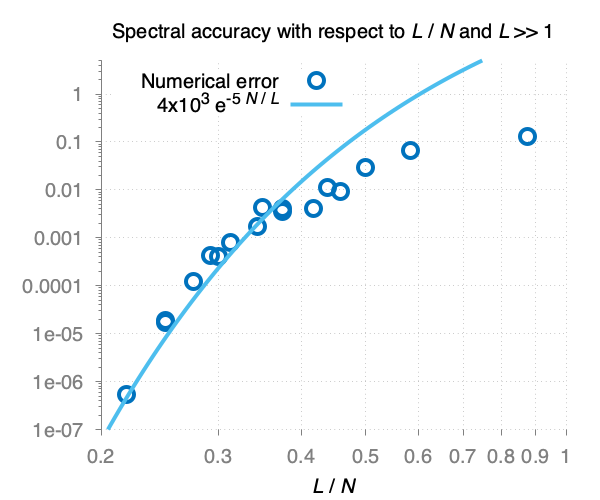}       
  \end{center}
\caption{{\bf Maxwellian interactions:}    error with respect to $L/N$,
  which illustrates the spectral accuracy when $L\ll N$.}
\label{fig:3}
\end{figure}

%
%
%

\subsection{Coulombian interactions ($\gamma=-3$)}
We choose $\gamma=-3$, which corresponds to the current analysis and
is the physically relevant case.   The initial condition is
$$
f_0(v) \,=\,
\frac{1}{S^2} \exp\left( -S\,\frac{(|v| - \sigma)^2}{\sigma^2}\right),
$$
 with $\sigma = 0.3$ and $S = 10$ and the integration time is $T
 =50$ with $\Delta t = 0.05$ with $N=48$ and $64$ in the computational domain $[-1.8,1.8]^3$.

 This test is used to compute the time evolution of the numerical solution \eqref{eq:fRn} and to compare the
 results with those obtained in \cite{PRT} and with a reference
 solution obtained with $N=128$ and a small time step.
 
We compute the time evolution of the distribution function and present
the numerical results in  Figures \ref{fig:t2-1} and \ref{fig:t2-2} where
computations were performed respectively with  $N=48$ and
$64$  collocation points and $L=1.8$.

On the one hand, we present in Figure \ref{fig:t2-1}, the time
evolution of the  entropy $\cH(f)$  and the fourth order moment of the distribution function
$$
M_4(t) = \int_{\R^3} f(t,v) \,\dD v.
$$
We also present the trend to equilibrium using the relative entropy
$\cH(f\,|\,\mu_f)$ and the $L^2$ norm of $f(t)-\mu_f$ where $\mu_f$ is
the Maxwellian equilibrium. 

We compare these quantities  with a reference solution obtained with
$N=128$ collocation points. We observe that with only $N=48$ in
each direction, we get an accurate solution on a large time intervall.  

\begin{figure}[ht!]
\begin{center}
  \begin{tabular}{cc}
  \includegraphics[width=8cm]{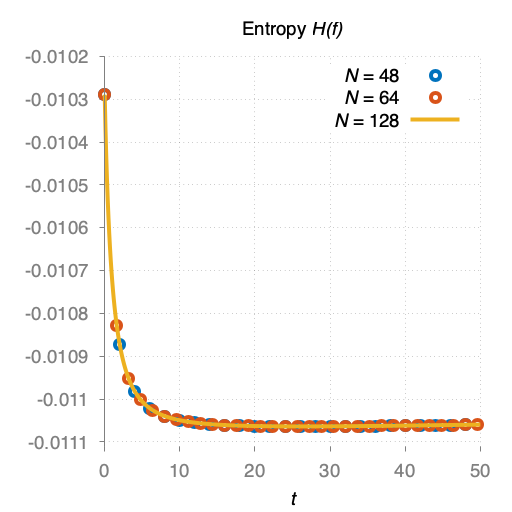} &
  \includegraphics[width=8cm]{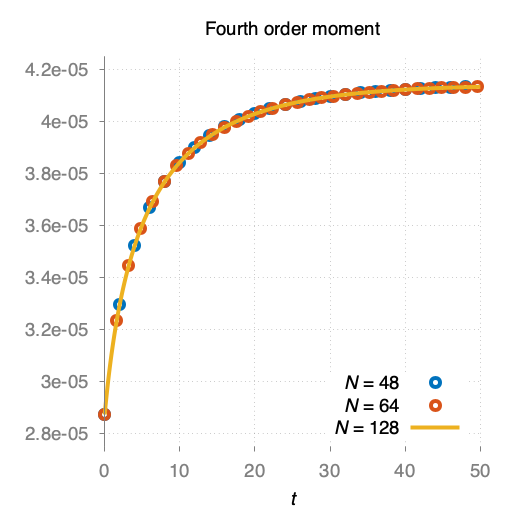}
      \\
    (a) $\cH(f)$ & (b) $M_4(f)$ \\
                   \includegraphics[width=8cm]{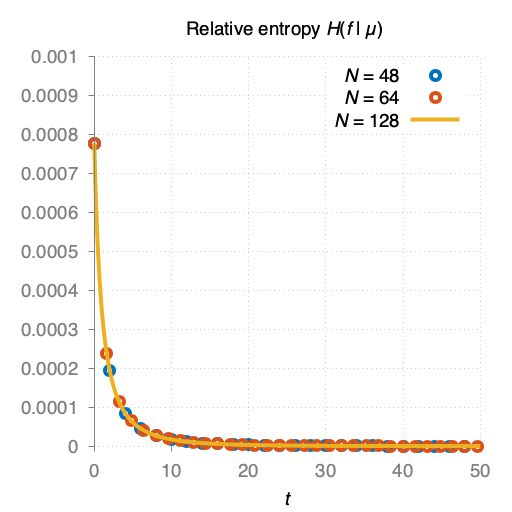} &
  \includegraphics[width=8cm]{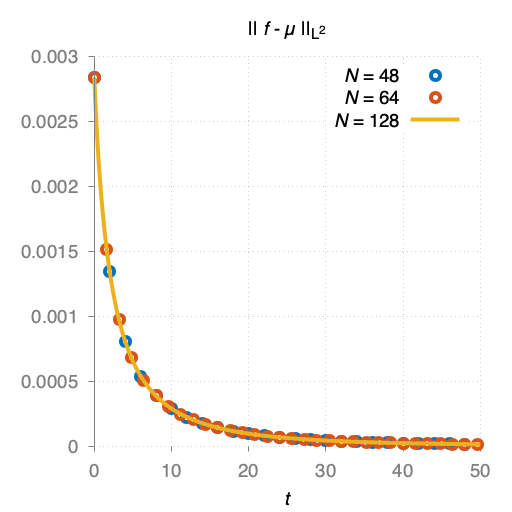}
      \\
 (c) $\cH(f\,|\,\mu_f)$ & (d) $\| f\,-\,\mu_f \|_{L^2}$  
\end{tabular}
  \end{center}
\caption{{\bf Coulombian interactions :} time evolution of the (a) 
  entropy $\cH(f)$, (b) fourth order moment $M_4$, (c) relative
  entropy $\cH(f\,|\;\mu_f)$ and (d) $L^2$ norm of $f\,-\,\mu_f$.}
\label{fig:t2-1}
\end{figure}

On the other hand, in Figure \ref{fig:t2-2} , we show the cross
section of the distribution function at times $t = 0$, $1$, $2$, $3$, $6$ and
$50$. The results are in good agreement with those presented in
\cite{PRT} even if the time scale does not corresponds since we
take all physical constants equal to one. Finally, Figure
\ref{fig:t2-2}-(b) represents the time evolution of the Fisher
information, which is known to be a Lyapunov functionnal for the
Landau equation \cite{GS}.

\begin{figure}[ht!]
\begin{center}
  \begin{tabular}{cc}
  \includegraphics[width=8.cm]{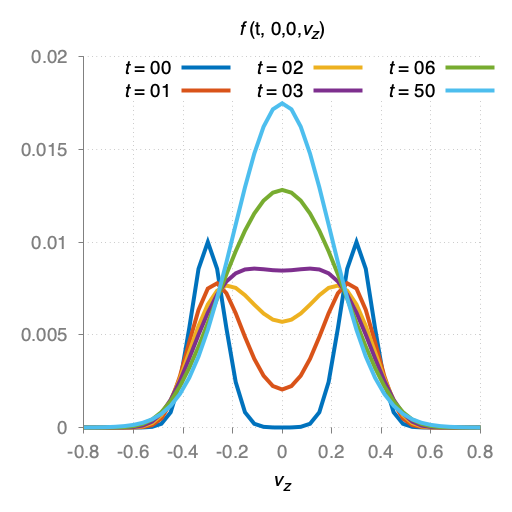} &
  \includegraphics[width=8.cm]{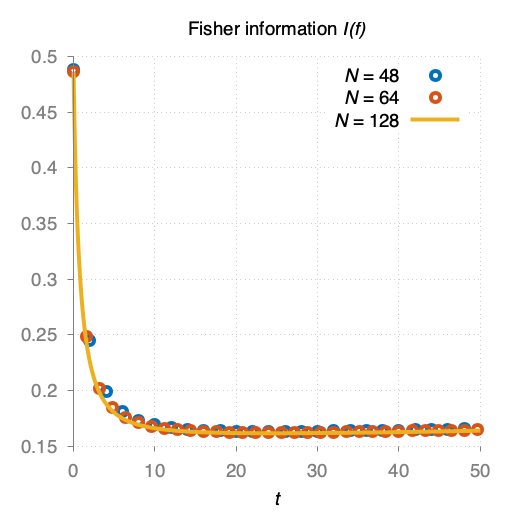}
   \\
   (a) & (b) 
\end{tabular}
  \end{center}
\caption{{\bf Coulombian interactions :} (a) time evolution the distribution
  function $f(t,0,0,v_z)$ for  $N=64$ at
  time $t = 0$,  $0.5$, $ 1$, $1.5$ and $3$ (b) time evolution
  of the Fisher information $\cI(f)$.}
\label{fig:t2-2}
\end{figure}

\section*{Acknowledgement}
The research of L.-B. He was supported by NSF of China under Grant No.11771236 and New Cornerstone Investigator Program 100001127.

\appendix
\section{Analytic results on the equation}
\label{sec:appendix1}
\setcounter{equation}{0}
\renewcommand\theequation{A\arabic{equation}}
\setcounter{figure}{0}
\setcounter{table}{0}

We begin with sharp bounds for the collision operator \cite[Proposition 2.1]{HJL}
 
\begin{thm}\label{coe}
Suppose that the nonnegative function $g$ satisfies $\|g\|_{L^1}>\delta,\|g\|_{L_2^1}\;+\;\cH(g)<\lambda$. Then there exists a constant $C(\delta, \lambda)$ such that
\ben
\|\nabla f\|_{L_{-3 / 2}^2}^2\;+\;\|\left(-\Delta_{\mathbb{S}^2}\right)^{\frac{1}{2}} f\|_{L_{-3 / 2}^2}^2 \;\le\; C(\delta, \lambda)\; \int(a * g): \nabla f \otimes \nabla f \dD v.
\een
\end{thm}
and \cite[Proposition 2.2]{HJL}.
\begin{prop}\label{UpperboundofQ}
  \begin{itemize}
    \item[(i).] Let $a_1$, $b_1\in[0,2]$, $\omega_1$, $\omega_2\in\R$, satisfying that $a_1+b_1=2$, $\omega_1+\omega_2=-1$. Then
      \ben
      \label{uppersob}
      \nr{\br{\cQ(g,h),f}}\;\ls\;
      \rr{\|g\|_{L^1_1}\;+\;\|g\|_{L^2_1}}\|h\|_{H^{a_1}_{\omega_1}}\|f\|_{H^{b_1}_{\omega_2}}.
      \een
    \item[(ii).] Let $a_2$, $b_2\in[0,1]$, $\omega_i\in\R$, $i=1$,
      $2$, $3$ and $4$, satisfying that $a_2+b_2=1$,
      $\omega_1+\omega_2=-3$ and $\omega_3+\omega_4=-2$. Then we have
      \ben
      \label{upperan}
    \nr{\br{\cQ(g,h),f}}&\ls_{\sss}&\rr{\|g\|_{L^1_3}\;+\,\|g\|_{L^2_{3}}}
    \;\rr{\|\rr{-\Delta_{\S^2}}^{\f{1}{2}}h\|_{L^2_{\omega_1}}\;+\;\|h\|_{H^1_{\omega_1}}}\rr{\|\rr{-\Delta_{\S^2}}^{\f{1}{2}}f\|_{L^2_{\omega_2}}\;+\;\|f\|_{H^1_{\omega_2}}}\notag
    \\
    &+& \LLO{g}\|h\|_{H^{a_2}_{\omega_3}}\;\|f\|_{H^{b_2}_{\omega_4}}.
    \een 
  \end{itemize}
\end{prop}

We also summarize several results on the Landau equation, including global well-posedness, propagation of regularity, monotonicity of Fisher information and global dynamics.

\begin{thm}\label{thm:coulombexist}
Let $f_0\in L^1_{2}$ be a nonnegative function with $\cH(f_0)<\infty$. 
\begin{itemize}
  \item [(1)]{\bf Well-posedness.}  Suppose that  $f_0 \in L_{2 m+1}^1
    \cap H_m^\mathsf{r}$   with $\mathsf{r} \in\left[-\frac{1}{2},
      0\right]$ and $m \geq \frac{9-6 \mathsf{r}}{4}$. \\
    Then there exists a global solution $f$ of the equation \eqref{1}
    satisfying that $f \in C\left([0, \infty) ; H_m^\mathsf{r}\right)
    \cap$ $L_{l o c}^2([0, \infty) ;
    H_{m-\frac{3}{2}}^{\mathsf{r}+1})$.  Furthermore, if $\mathsf{r}>1/2$, then the solution is unique.
  \item [(2)] {\bf Smoothing estimates.}  Let  $f\in C([0,\infty);H^{\mathsf{r}}_m)\cap L^2_{loc}([0,\infty);H^{\mathsf{r}+1}_{m-\f32})$ 
  be a global solution of the equation \eqref{1} with
  $\mathsf{r}\in[-\f12,0]$ and $m\geq \f{9-6\mathsf{r}}{4}$,
  satisfying $\inf_{t\in[0,\infty)}\|f(t)\|_{L^1}>0$ and
  $\sup_{t\in[0,\infty)}\{\|f(t)\|_{L^1_2} +\cH(t)\}<\infty$. Then 
  for any $\nn\geq\mathsf{r}$, it holds that
  \ben\label{SMesti}\|f(t)\|_{H^\nn_{m-\f{3}{2}\nn+\f{3\mathsf{r}}{2}}}\ls t^{-\f{\nn}{2}+\f{\mathsf{r}}{2}}\mathbf{1}_{t\leq1}+C(t)\mathbf{1}_{t>1},\quad \forall t\in(0,\infty),\een
  where \( C(t) \) is a locally uniformly bounded function depending
  only on $\|f_0\|_{H^\mathrm{r}_m}$, $\mathsf{r}$, $m$, $\nn$, $\inf_{t\in[0,\infty)}\|f(t)\|_{L^1}$ and $\sup_{t\in[0,\infty)}\{\|f(t)\|_{L^1_2}+\cH(f(t))\}$.  
  \item [(3)]{\bf  Propagation of regularity near initial time.}  For any $\nn\geq2, \ell\geq3$, if $\|f_0\|_{H^\nn_\ell}<\infty$, then there exists a $T>0$ depending solely on $\nn,\ell$ and $\|f_0\|_{H^\nn_\ell}$ such that the weak solution to \eqref{1} with initial data $f_0$ satisfies
  \ben\label{stpReg}
  \sup_{t\in[0,T]}\|f(t)\|_{H^\nn_\ell}\ls_{\nn,\ell} \|f_0\|_{H^\nn_\ell}.
  \een
  \item [(4)] {\bf Fisher information.} Let $f$ be a classical solution to the spatially homogeneous Landau equation \eqref{1} with Coulomb potential and with initial data $f_0$, then the Fisher information given
    by $$
    \cI(f)\; =\;\int \f{|\na f|^2}{f} \, \dD v,
    $$
    is a  monotone decreasing  function in time.
  \item [(5)] {\bf Propagation of $L^1$ moment and longtime behavior.}
    Consider any global $H$ - or weak solution $f$ to the spatially homogeneous Landau equation \eqref{1} with Coulomb potential and with initial data $f_0$ under the Assumption \ref{TAES} - \eqref{hyp:0}. Assume that $f_0 \in L_{\ell}^1\left(\mathbb{R}^3\right)$ with $\ell>\frac{19}{2}$. Then, for any positive $\beta<\frac{2 \ell^2-25 \ell+57}{9(\ell-2)}$, there exists $C_\beta>0$ depending only on $\beta$ and $f_0$, such that
\ben
  \forall t \geq 0, \quad \|f-\mu\|_{L^1}^2 \;\leq\; \frac{C_\beta}{(1+t)^{\beta}}.
\een
Moreover, it holds that
\ben
  \|f(t)\|_{L^1_\ell}  \;\ls\;   \|f_0\|_{L^1_\ell} \;+\;  \ell^{2(\ell-6)+3}\; t\;.
\een
\end{itemize}
\end{thm}

Results $(1)$, $(2)$ and $(3)$ are respectively established in
\cite[Corollary 1.1]{HJL}, \cite[Theorem 1.2]{HJL} and
\cite[Proposition 3.1]{HJL} while  $(4)$
 is proved in \cite{GS} and  $(5)$ is demonstrated in \cite[Theorem 2 and Lemma 8]{CDH}.

 \section{Technical lemmas}
 \label{sec:appendix2}
 \setcounter{equation}{0}
\renewcommand\theequation{B\arabic{equation}}
\setcounter{figure}{0}
\setcounter{table}{0}

 We begin with a profile to characterize the weighted Sobolev spaces proved in \cite[Theorem 5.1]{HE18}.

\begin{lem}\label{thm:dyadic} 
  Let $m$, $l \in \mathbb{R}$. Then for $f \in H_l^m$, we have
\ben
  \sum_{k=-1}^{\infty} 2^{2 k l}\left\|\cP_k f\right\|_{H^m}^2 \;\sim_{m,l}\;\|f\|_{H_l^m}^2,
\een
where $\cP_j$ is a localized operator defined as
\ben
\cP_{-1} f(x)\;:=\;\psi(x) f(x), \quad \cP_j f(x)\;=\;\varphi(2^{-j}
x) f(x),\quad j \geq 0,
\een
with $\psi\in C_0^{\infty}\left(\cB(0,4/3)\right)$, $\varphi\in C_0^{\infty}(\cB(0,8/3)\backslash \cB(0,3/4))$ satisfying 
\beq
0 \;\leq\; \psi,\; \varphi \;\leq\; 1, \quad \psi(\xi)+\sum_{j \geq 0} \varphi(2^{-j} \xi)=1, \quad \xi \in \mathbb{R}^3.
\eeq
\end{lem}
As a result, we can prove the following interpolation inequality.
\begin{lem}\label{thm:interpolationSobolev}
Let $m$, $l\in\R$ and $\theta\in (0,1)$, satisfying that  $m=m_1\theta+m_2(1-\theta)$, $l=l_1\theta+l_2(1-\theta)$. Then we have
\ben
\label{eq:interpolationSobolev}
\|f\|_{H^m_l} \;\ls\; \Hab{f}{m_1}{l_1}^{\theta}\;\Hab{f}{m_2}{l_2}^{1-\theta}.
\een
\end{lem}

\begin{proof}
  Thanks to  Lemma \ref{thm:dyadic}, since $m=m_1\theta+m_2(1-\theta)$
  and $l=l_1\theta+l_2(1-\theta)$,  we have
\begin{eqnarray*}
\|f\|_{H^m_l}^2&\ls & \ds\sum_{k=-1}^{\infty} 2^{2 k l}\left\|\cP_k
                            f\right\|_{H^m}^2\\
  &\ls & \ds\sum_{k=-1}^{\infty} 2^{2 k l}\left\|\cP_k f\right\|_{H^{m_1}}^{2\theta}\left\|\cP_k f\right\|_{H^{m_2}}^{2(1-\theta)}
  \\
               &=&\ds\sum_{k=-1}^{\infty} \left( 2^{2 k l_1}\left\|\cP_k
                   f\right\|_{H^{m_1}}^{2} \right)^{\theta}\left( 2^{2
                   k l_2}\left\|\cP_k f\right\|_{H^{m_2}}^{2}
                   \right)^{1-\theta}
  \,\ls\;  \ds\Hab{f}{m_1}{l_1}^{2\theta}\Hab{f}{m_2}{l_2}^{2(1-\theta)},
\end{eqnarray*}
which completes the proof.
\end{proof}

The following  inequalities will be frequently employed in our
analysis of weighted Sobolev spaces and the anisotropic structure of
the collision operator
%
\cite[ Lemma 5.7]{HE18}.
 \begin{lem}
   \label{thm:fs2}
   Suppose $f \in H_s^s\left(\mathbb{R}^3\right)$ with $s \geq 0$. Then 
  \ben
  \label{eq:h1s}
\|f\|_{H^1_{-3/2}}\ls \|f\|_{\H1A}\ls \|f\|_{H^1_{-1/2}},
\een
where the space $\H1A$ is defined in \eqref{eq:defh1s}
\end{lem}

Next we prove 
\begin{lem}
  \label{WN2lHNl}
  Let $n\in \N$ and $l\in\R$. Then for smooth
  function $f$, it holds that
  $$
  \|f\|_{H^n_l}\;\sim\;  \|f\|_{W^{n,2}_l}.
  $$
\end{lem}
\begin{proof}
  We apply an inductive method to prove the desired result. By the definition, the equivalence holds for $n=0$. Let us assume that it holds for $n-1$ with $n\ge1$. To prove the result holds for $n$, we observe that 
  \beno
  \|f\|_{H^N_l}^2\,=\, \|\br{D}^n \br{v}^l f\|_{L^2}^2\,\sim\; \|
  \br{D}^{n-1} \br{v}^l f\|_{L^2}\;+\;\sum_{i=1}^3\| R_i\pa_{v_i}(
  \br{D}^{n-1} \br{v}^l f)\|_{L^2},
  \eeno 
where $R_i$ denotes the Riesz transform associated with the multiplier
$\xi/|\xi|$.
Using the fact that $\|f\|_{H^{n-1}_l}\sim_{N,l} \|f\|_{W^{n-1,2}_l}$, we have 
\begin{eqnarray*}
  \|f\|_{H^n_l}^2 &=& \|\br{D}^n \br{v}^l f\|_{L^2}^2
  \\
  &\ls& \|  \br{D}^{n-1} \br{v}^l
  f\|_{L^2}+\sum_{i=1}^3(\|\br{D}^{n-1} \br{v}^l
  \pa_{v_i}f)\|_{L^2}+\|\br{D}^{n-1} (\pa_{v_i}\br{v}^l) f)\|_{L^2})
  \\
  &\ls& \|f\|_{H^{n-1}_l}+\|\na f\|_{H^{n-1}_l} \;\ls\;
  \|f\|_{W^{n,2}_l}.
  \end{eqnarray*} 
Here we use the fact that $\|\br{D}^{n-1} (\pa_{v_i}\br{v}^l)
f)\|_{L^2}\ls \|f\|_{H^{n-1}_l}$.
We refer to \cite[ Lemma 9.1]{HJL} for detailed proof.
\\
To prove the inverse inequality, we notice that 
\begin{eqnarray*}
  \|f\|_{W^{n,2}_l} &\ls& \|f\|_{W^{n-1,2}_l}\;+\;\|\na
                          f\|_{W^{n-1,2}_l}
  \\
                    &\ls& \|f\|_{H^{n-1}_l}\;+\;\|\br{D}^{m-1} \br{v}^l (\na f)\|_{L^2}
  \\
                    &\ls& \|f\|_{H^n_l}\;+\;\|\br{D}^{m-1}\na \br{v}^l
                          f\|_{L^2}\;+\;\|\br{D}^{m-1}(\na \br{v}^l)
                          \; f\|_{L^2}
  \\
                    &\ls& |f\|_{H^n_l}^2.
                          \end{eqnarray*}
In the last inequality,  we again used the fact that $\|\br{D}^{n-1} (\pa_{v_i}\br{v}^l) f)\|_{L^2}\ls \|f\|_{H^{n-1}_l}$.
 \end{proof}

Finally we state a theorem on the convolution inequality proved in \cite[Lemma 3.3]{C15}.

\begin{thm}
  \label{Kalpha}
  Suppose that $0<\alpha<d$ and $\ell_0$, $\ell_{1}$, $\ell_{2}$,
  $n_0$, $n_1$, $n_2\in\R$ satisfying that
  $$
  |\ell_{0}|\leq\alpha, \ell_{1}+\ell_{2}\;=\;-\ell_{0},
  n_{1}+n_{2}\;=\;-n_{0}.
  $$
  Let us define
\ben
K_\alpha(f,g,h)\;:=\;\iint_{\mathbb{R}^d\times\mathbb{R}^d}|v-v_*|^{-\alpha}f_*\,g\;h\,\dD v_*\,\dD v.
\een
Then, the following estimates hold:
  \begin{itemize}
    \item[(1)] For any $\sigma\geq 0$ with $2\sigma<d$ and $2(\alpha-\sigma)<d$, we have
$$
K_\alpha(f,g,h) \;\lesssim\;\|f\|_{L^1_{\ell_0}}\|g\|_{L^2_{\ell_1}}\|h\|_{L^2_{\ell_2}}\;+\;\|f\|_{L^1_{n_0}}\|\langle v\rangle^{n_1}g\|_{{\dot{H}}^{\alpha-\sigma}}\|\langle v\rangle^{n_2}h\|_{{\dot{H}}^{\sigma}}.
$$
\item[(2)]  For any $0\le \sigma_1<d/2$, $0\le \sigma_2<d/2 $ with $0<\sigma_1+\sigma_2< \alpha$, we have
$$
K_\alpha(f,g,h)\;\lesssim\;\|f\|_{L_{\ell_0}^1}\|g\|_{L_{\ell_1}^2}\|h\|_{L_{\ell_2}^2}\;+\;\|f\|_{L^{\frac{d}{d-\alpha+\sigma_1+\sigma_2}}_{n_0}}\|g\|_{H^{\sigma_1}_{n_1}}\|h\|_{H^{\sigma_2}_{n_2}}.
$$
\item[(3)] For any $p>\f{d}{d-\alpha}$, we have
$$
K_\alpha(f,g,h)\;\lesssim\,\|f\|_{L_{\ell_0}^1}\;\|g\|_{L_{\ell_1}^2}\;\|h\|_{L_{\ell_2}^2}\;+\;\|f\|_{L^{p}_{n_0}}\;\|g\|_{L^2_{n_1}}\;\|h\|_{L^2_{n_2}}.
$$
  \end{itemize} 
where the constants depend only on $d$, $\alpha$, $n_0$, $\sigma$,
$\sigma_1$, $\sigma_2$ and $p$, don't depend on the choice of
$\ell_0$, $\ell_1$, $\ell_2$, $n_1$ and $n_2$.
\end{thm}

 \section{Proof of Proposition \ref{prop:betalm} }
 \label{sec:appendix3}
 \setcounter{equation}{0}
\renewcommand\theequation{C\arabic{equation}}
\setcounter{figure}{0}
\setcounter{table}{0}
  By  definition of periodic Landau collision operator \eqref{defQp}, one has
\begin{eqnarray*}
 \cQ_\#(e^{i\f{\pi}{L}l\cdot v},e^{i\f{\pi}{L}m\cdot v}) &=&\na\cdot
\int_{\cD_L}
a^L(z)\left(e^{i\f{\pi}{L}l\cdot
(v+z)}\,
\na
e^{i\f{\pi}{L}m\cdot
v}-\na
e^{i\f{\pi}{L}l\cdot
(v+z)}\,
e^{i\f{\pi}{L}m\cdot
v}\right)\,\dD z
\\
&=&\int_{\cD_L}a^L(z)\left(m\otimes m-l\otimes
    l\right)\,(i\f{\pi}{L})^2\, e^{i\f{\pi}{L}(l\cdot(v+z)+m\cdot v)}\,\dD z
  \\
  &=& \f{\pi^2}{L^2}\int_{\cD_L}a^L(z):(l+m)\otimes
      (l-m)\,e^{i\f{\pi}{L}l\cdot z} \;\dD z \ e^{i\f{\pi}{L}(l+m)\cdot v},
\end{eqnarray*}
which implies that
\beno
\beta(l,m)\,=\;\f{\pi^2}{L^2}\int_{\cD_L}a^L(z):(l+m)\otimes (l-m)\,e^{i\f{\pi}{L}l\cdot z}\,dz.\eeno 
From this latter expression  and applying a change of variable, we derive that
\begin{eqnarray*}
\beta(l,m)&=&\f{\pi^2}{L^2}\int_{\cB(0,\pi)}\left|\f{L}{\pi}z\right|^{-1}\Pi(z):(l+m)\otimes
              (l-m)\,e^{il\cdot z}\left( \f{L}{\pi} \right)^3 \dD z
\\
\;&=& \int_{\cB(0,\pi)}|z|^{-1}\, \Pi(z):(l+m)\otimes (l-m)\,e^{il\cdot
  z}\,\dD z.
\end{eqnarray*}

We first consider the case where  $|l|>0$ and  introduce
\ben
e_1 \;:=\;\f{l}{|l|},\quad e_3\;:=\;\f{l\times m}{|l\times m|},\quad e_2\;:=\;e_3\times e_1,
\een
and when $|m|=0$ or $|l\times m|=0$, we pick $e_3$ to be a unit vector perpendicular to $e_1$. Thus, the vectors $e_1$, $e_2$  and $e_3$ form an orthogonal coordinate system in $\mathbb{R}^3$. By definition we get that $l=|l|e_1\in span\{e_1\},m\in span\{e_1,e_2\}$, thus $l+m,l-m\in span\{e_1,e_2\}$. Setting
\ben
\label{l+ml-m}
l+m\;=\; a_1\; e_1\;+\; a_2\;e_2,\quad \ l-m\;=\; b_1\; e_1\;+\;b_2\;e_2,
\een
it yields that
\ben\label{eq:a12}
a_1\;=\;\f{l\cdot (l+m)}{|l|},\quad a_2\;=\;\f{|l\times (l+m)|}{|l|}\;=\;\f{|l\times m|}{|l|},
\een
and
\ben\label{eq:b12}
b_1\;=\;\f{l\cdot (l-m)}{|l|}, \quad b_2\;=\;-\f{|l\times (l-m)|}{|l|} \;=\;-\f{|l\times m|}{|l|}\;=\;-a_2.
\een
Now we introduce the  polar coordinate: for $r\ge 0,\theta\in[0,\pi),\phi\in[0,2\pi)$, suppose that
\beq
z\;=\; r\cos\theta\, e_1\;+\;r\sin\theta \cos\phi\, e_2\;+\; r\sin\theta \sin\phi\, e_3.
\eeq
Then $\dD z=r^2\,\sin\theta\, \dD r\, \dD \theta\, \dD \phi$. From this together with \eqref{Pizpq} and \eqref{l+ml-m}, we have 
\begin{eqnarray*}
|z|^2\,\Pi(z):(l+m)\otimes (l-m) &=& (z\times (l
                                 +m))\cdot (z\times (l-m))
  \\
&=& r^2\left(\sin^2\theta\sin^2\phi\, a_2\,b_2+\sin^2\theta
    \,a_1\,b_1+\cos^2\theta\, a_2\,b_2\right)
  \\
  &-& r^2 \left(\sin\theta\cos\theta\cos\phi
    \,(a_1\,b_2+a_2\,b_1)\right)
  \\
  &:=& r^2\; I(\theta,\phi).
\end{eqnarray*}
Since $e^{il\cdot z}=e^{i|l|r\cos\theta}$, we  derive that
\beno
\beta(l,m)&=&  \int_0^\pi \dD r\int_0^\pi \dD \theta\int_0^{2\pi} \dD \phi \ r^{-1}I(\theta,\phi)e^{i|l|r\cos\theta}\,r^2\sin\theta\\
&=& \int_0^\pi \dD r\int_0^\pi \dD \theta \left(\pi\sin^2\theta\, a_2\,b_2+2\pi\sin^2\theta \,a_1\,b_1+2\pi\cos^2\theta \,a_2\,b_2\right).
\eeno

Let $x=|l|\,r>0$,  we may reduce the computation to  the following integrals:
\beq
I_1(x)\;=\;\int_0^\pi \sin^2\theta\, e^{ix\cos\theta}\sin\theta\, \dD\theta,
\quad{\rm and}\quad \ I_2(x)\;=\;\int_0^\pi \cos^2\theta\, e^{ix\cos\theta}\sin\theta \,\dD\theta.
\eeq
 By the change of variable $t:=\cos\theta$, we are led to that
\beq
I_1(x)\;=\;\int_{-1}^{1}(1-t^2)\,e^{ixt}\,\dD t\;=\;\f{4}{x^3}(\sin x-x\cos x),
\eeq
and
\beq
I_2(x)\;=\;\int_{-1}^{1}t^2\,e^{-ixt}\,\dD t\;=\;\f{2}{x^3}((x^2-2)\sin x+2x\cos x).
\eeq
From these, we further derive that 
\begin{eqnarray*}
\beta(l,m)&=&\pi \int_0^\pi
\left(r\,((a_2\,b_2+2a_1\,b_1)\,I_1(|l|\,r)+2a_2\,b_2\,I_2(|l|\,r))\right)\dD r\\
&=& \f{\pi}{|l|^{2}}\int_0^{|l|\pi}\left(r\,
  ((a_2\,b_2+2a_1\,b_1)\,I_1(r)+2a_2\,b_2\,I_2(r))\right)\dD r
\\
&=&  \f{\pi}{|l|^{2}} \left( a_1\, b_1\,F_1(|l|\,\pi)\;+\;a_2\, b_2\,F_2(|l|\,\pi) \right),
\end{eqnarray*}
where $F_1$ and $F_2$ are defined by
$$
\left\{\begin{array}{l}
\label{eq:F1F2}
\ds F_1(x)\;=\;2\int_0^{x}r\,I_1(r)\,\dD r\;=\; 8\int_0^{x}r^{-2}(\sin
         r-r\cos r)\,\dD r\;=\;8\left(1-\f{\sin x}{x}\right),
\\[0.9em]
\ds F_2(x) \;=\;\int_0^{x} r\,(I_1(r)+2I_2(r))\,\dD r\;=\;
         4\int_0^{x}r^{-2}((r^2-1)\sin r+r\cos r)\,\dD r\;=\;4\left(\f{\sin x}{x}-\cos x\right).
       \end{array}\right.
     $$
From \eqref{eq:a12} and \eqref{eq:b12}, it is easy to see that
\ben\label{eq:ab12}
a_1\,b_1\,=\;\f{(l\cdot (l+m))(l\cdot (l-m))}{|l|^2} \;=\;\f{|l|^4-(l\cdot m)^2}{|l|^2},\quad a_2\,b_2\;=\;-\f{|l\times m|^2}{|l|^2}.
\een
Substituting this latter expression for $F_1$,  $F_2$ and \eqref{eq:ab12} into $\beta(l,m)$, we  have 
\beq
\beta(l,m)\;=\;\f{\pi}{|l|^4}\left( 8(|l|^4-(l\cdot m)^2) \left(1-\f{\sin (|l|\pi)}{|l|\pi}\right)-4|l\times m|^2\left(\f{\sin (|l|\pi)}{|l|\pi}-\cos (|l|\pi)\right) \right).
\eeq

Finally, we consider the case  where $|l|=0$, then we have
$$
\beta(0,m)\;=\;- \int_{\cB(0,\pi)}|z|^{-1}\,\Pi(z):m\otimes m \,dz=-
\int_{\cB(0,\pi)}|z|^{-3}|z\times m|^2\,\dD z.
$$
If $|m|=0$, then  $\beta(0,0)=0$; whereas if $|m|\neq 0$, let us set
$e_1=\f{m}{|m|}$ and $z=r\cos\theta\, e_1+r\sin\theta\, \sigma$ where
$\sigma\in \S^1 \,e_2$ with some unit vector $e_2$ perpendicular to $e_1$, hence it yields
\begin{eqnarray*}
\beta(0,m)&=&- \int_0^{\pi} \dD r\int_0^\pi \dD\theta \int_{\S^1}\dD \sigma\  r^{-3}\,(r|m|\sin \theta)^2\,r^2\sin\theta
  \\
  &=& -2\pi|m|^2  \int_0^\pi \f{4}{3}r^{ }\,dr\;=\; -\f{4|m|^2}{3 } \pi^3,
\end{eqnarray*}
which completes the proof.

\end{document}